\pdfoutput=1

\newif\ifCmp\Cmptrue
\newcount\FmtChoice
\ifcase\FmtChoice
\documentclass{interact}
\or
\documentclass[a4paper,ngerman,american,reqno]{amsart}
\usepackage{a4wide}
\or
\documentclass[english]{jnsao}
\fi

\usepackage{xcolor}
\usepackage{xspace}
\usepackage{mathtools}
\usepackage{microtype}
\usepackage{bm}
\usepackage{empheq}
\usepackage{cancel}
\usepackage{tikz}
\usetikzlibrary{matrix}
\usetikzlibrary{arrows.meta}
\usepackage{ao-math-std}
\usepackage{ao-math-symbols}
\usepackage{centernot}
\usepackage{rotating}
\usepackage[nameinlink,capitalize]{cleveref}

\usepackage[numbers,sort&compress]{natbib}
\bibpunct[, ]{[}{]}{,}{n}{,}{,}

\ifnum\FmtChoice<2
\newtheorem{theorem}{Theorem}
\newtheorem{lemma}[theorem]{Lemma}

\theoremstyle{definition}
\newtheorem{definition}[theorem]{Definition}

\theoremstyle{remark}
\newtheorem{remark}[theorem]{Remark}
\fi

\DeclareMathOperator\sign{sign}

\DeclareMathOperator\diag{diag}

\ifcase0
\newcommand\minst[2][]{\min_{#1} \ #2 \qstq}
\newcommand\sminst[2][]{\smash[b]{\min_{#1}} \ #2 \qstq}
\or
\newcommand\minst[2][]{\min_{#1} \quad &#2 \\ \stq}
\fi

\newcommand\tcones{tangent cones\xspace}
\newcommand\tcone{tangent cone\xspace}
\newcommand\TCone{Tangent Cone\xspace}

\newcommand{\compl}{\ensuremath\perp}
\newcommand{\R}{\mathbb R}

\newcommand\setE{\mathcal E}
\newcommand\setI{\mathcal I}
\newcommand\setZ{\mathcal Z}
\newcommand\setA{\mathcal A}

\newcommand\setN{\mathcal N}
\newcommand\setF{\mathcal F}
\newcommand\setT{\mathcal T}

\newcommand\setP{\mathcal P}
\newcommand\setPt{\mathcal P^t}
\newcommand\setPw{\mathcal P^w}
\newcommand\setPtw{\mathcal P^{t,w}}

\newcommand\setU{\mathcal U}
\newcommand\setUt{\setU^t}
\newcommand\setUw{\setU^w}

\newcommand\setV{\mathcal V}
\newcommand\setVt{\setV^t}
\newcommand\setVw{\setV^w}

\newcommand\setD{\mathcal D}
\newcommand\setDt{\setD^t}
\newcommand\setDw{\setD^w}

\newcommand{\Cspace}{C}
\newcommand{\Cd}{\Cspace^d}
\newcommand{\Cabs}{\Cd_{\text{abs}}}

\newcommand{\Dom}{D}
\newcommand\Domx{\Dom^x}
\newcommand\Domz{\Dom^{\abs{z}}}
\newcommand\Domt{\Dom^t}
\newcommand{\Domzt}{\Dom^{\abs{\zt}}}
\newcommand\Domxz{\Dom^{x, \abs{z}}}
\newcommand\Domtzt{\Dom^{t, \abs{\zt}}}

\newcommand\Fabs{\setF\tsb{abs}}
\newcommand\Tabs{\setT\tsb{abs}}
\newcommand\Tlinabs{\Tabs\tsp{lin}}
\newcommand\Feabs{\setF\tsb{e-abs}}
\newcommand\Teabs{\setT\tsb{e-abs}}
\newcommand\Tlineabs{\Teabs\tsp{lin}}

\newcommand\Fsig{\setF_{\Sigt}}
\newcommand\Tsig{\setT_{\Sigt}}
\newcommand\Tlinsig{\Tsig\tsp{lin}}
\newcommand\Fesig{\setF_{\Sigtw}}
\newcommand\Tesig{\setT_{\Sigtw}}
\newcommand\Tlinesig{\Tesig\tsp{lin}}

\newcommand\Fp{\setF_{\setPt}}
\newcommand\Tp{\setT_{\setPt}}
\newcommand\Tlinp{\Tp\tsp{lin}}
\newcommand\Fep{\setF_{\setPtw}}
\newcommand\Tep{\setT_{\setPtw}}
\newcommand\Tlinep{\Tep\tsp{lin}}

\newcommand\Fmpec{\setF\tsb{mpcc}}
\newcommand\Fempec{\setF\tsb{e-mpcc}}
\newcommand\Tmpec{\setT\tsb{mpcc}}
\newcommand\Tcompl{\setT_\compl}
\newcommand\Tempec{\setT\tsb{e-mpcc}}
\newcommand\Tlinmpec{\Tmpec\tsp{lin}}
\newcommand\Tlinempec{\Tempec\tsp{lin}}

\newcommand{\ce}{c_{\setE}}
\newcommand{\ci}{c_{\setI}}
\newcommand{\cz}{c_{\setZ}}
\newcommand\cA{c_{\setA}}

\newcommand\bcE{\_c_\setE}
\newcommand\bcZ{\_c_\setZ}

\newcommand\zt{z^t}
\newcommand\zw{z^w}
\newcommand\hzt{\^z^t}
\newcommand\hzw{\^z^w}

\newcommand\alpt{\alpha^t}
\newcommand\alpw{\alpha^w}
\newcommand\sigt{\sigma^t}
\newcommand\sigw{\sigma^w}
\newcommand\Sigt{\Sigma^{t}}
\newcommand\Sigw{\Sigma^{w}}
\newcommand\Sigtw{\Sigma^{t,w}}

\newcommand\calp[1][]{\card{\alpha#1}}

\newcommand\csig[1][]{\card{\sigma#1}}

\newcommand\lame{\lambda_{\setE}}
\newcommand\lami{\lambda_{\setI}}
\newcommand\lamz{\lambda_{\setZ}}

\newcommand\ut{u^t}
\newcommand\uw{u^w}
\newcommand\hut{\^u^t}
\newcommand\huw{\^u^w}

\newcommand\vt{v^t}
\newcommand\vw{v^w}
\newcommand\hvt{\^v^t}
\newcommand\hvw{\^v^w}

\newcommand\Lc{\mathcal L_\compl}

\newcommand\muu{\mu_u}
\newcommand\muv{\mu_v}

\newenvironment{defarray}[2][l]
{\left\{\,#2\,\left|\def\arraystretch{\defarraystretch}\begin{array}{#1}}
{\end{array}\right.\!\right\}}
\newenvironment{ccases}[1][ll]
{\left\{\!\begin{array}{#1}}{\end{array}\!\right\}}

\newcommand\delpos[1]{\langle\delta#1\rangle^+}
\newcommand\delneg[1]{\langle\delta#1\rangle^-}

\ifnum\FmtChoice=1
  \newenvironment{amssidewaysfigure}
  {\begin{sidewaysfigure}\vspace*{.65\textwidth}\begin{minipage}{\textheight}\centering}
  {\end{minipage}\end{sidewaysfigure}}
\fi

\newcommand\abstracttext{%
  This work continues an ongoing effort to compare
  non-smooth optimization problems in abs-normal form to Mathematical Programs with Complementarity Constraints (MPCCs).
  We study general Nonlinear Programs
  with equality and inequality constraints in abs-normal form,
  so-called Abs-Normal NLPs,
  and their relation to equivalent MPCC reformulations.
  We introduce the concepts of Abadie's and Guignard's kink qualification and prove relations to MPCC-ACQ and MPCC-GCQ for the counterpart MPCC formulations.
  Due to non-uniqueness of a specific slack reformulation suggested in
  \cite{Hegerhorst_Steinbach:2019}, the relations are non-trivial.
  It turns out that constraint qualifications of Abadie type are preserved. We also prove the weaker result
  that equivalence of Guginard's (and Abadie's) constraint qualifications
  for all branch problems hold, while the question of GCQ preservation
  remains open.
  Finally, we introduce M-stationarity and B-stationarity concepts
  for abs-normal NLPs and prove
  first order optimality conditions corresponding to MPCC counterpart formulations.
}
\newcommand\keywordlist{%
  Non-smooth NLPs,
  abs-normal form,
  MPCCs,
  Abadie and Guignard type constraint qualifications,
  optimality conditions%
}
\newcommand\amscodelist{%
  90C30, 
  90C33, 
  90C46
}
\newcommand\addrIfAM{Leibniz Universit\"at Hannover,
  Institut f\"ur Angewandte Mathematik,
  Welfengarten~1, 30167 Hannover, Germany}
\newcommand\addrIMO{Technische Universit\"at Carolo-Wilhelmina zu Braunschweig,
  Institut f\"ur Mathematische Optimierung, Universit\"ats\-platz~2,
  38106 Braunschweig, Germany}

\ifcase\FmtChoice\or\or

\title{Relations Between Abs-Normal NLPs and MPCCs.
  Part 2: Weak Constraint Qualifications}

\date{\ISOToday}

\author{L.C. Hegerhorst-Schultchen%
  \thanks{\addrIfAM. \email{hegerhorst@ifam.uni-hannover.de}}
  \and
  C. Kirches\thanks{\addrIMO. \email{c.kirches@tu-bs.de}}
  \and
  M.C.~Steinbach\thanks{\addrIfAM. \email{mcs@ifam.uni-hannover.de}}
}
\shortauthor{Hegerhorst-Schultchen, Kirches, Steinbach}

\acknowledgements{C.~Kirches discloses support by the German Federal Ministry of Education and Research through grants
n\textsuperscript{o} 05M17MBA-MoPhaPro, 05M18MBA-MOReNet, and 01/S17089C-ODINE,
and by Deutsche Forschungsgemeinschaft (DFG) through Priority Programme 1962, grant Ki1839/1-2.}

\manuscriptcopyright{{\copyright} the authors}
\manuscriptlicense{CC-BY-SA 4.0}

\manuscriptsubmitted{2020-07-30}
\manuscriptaccepted{2021-02-12}
\manuscriptvolume{2}
\manuscriptnumber{6673}
\manuscriptyear{2021}
\manuscriptdoi{10.46298/jnsao-2021-6673}

\fi

\begin{document}

\ifcase\FmtChoice
\author{
  \name{L.C.~Hegerhorst-Schultchen\textsuperscript{a}$^\ast$%
    \thanks{Corresponding author. E-Mail {\tt hegerhorst@ifam.uni-hannover.de}.}
    \and
    C. Kirches\textsuperscript{b}
    \and
    M.C.~Steinbach\textsuperscript{a}}
  \affil{\textsuperscript{a}\addrIfAM.}
  \affil{\textsuperscript{b}\addrIMO.}
}
\title{Relations Between Abs-Normal NLPs and MPCCs\\
  Part 2: Weak Constraint Qualifications}
\date{\today}
\maketitle
\begin{abstract}
  \abstracttext
\end{abstract}
\begin{keywords}
  \keywordlist
\end{keywords}
\begin{amscode}
  \amscodelist
\end{amscode}
\or
\author{L.C.~Hegerhorst-Schultchen}
\address{L.C. Hegerhorst-Schultchen, \addrIfAM.}
\email{hegerhorst@ifam.uni-hannover.de}
\author{C. Kirches}
\address{C. Kirches, \addrIMO.}
\email{c.kirches@tu-bs.de}
\author{M.C.~Steinbach}
\address{M.C.~Steinbach, \addrIfAM.}
\email{mcs@ifam.uni-hannover.de}
\title[Relations Between Abs-Normal NLPs and MPCCs. Part 2: Weak CQs]
{Relations Between Abs-Normal NLPs and MPCCs\\
  Part 2: Weak Constraint Qualifications}
\begin{abstract}
  \abstracttext
\end{abstract}
\keywords{\keywordlist}
\subjclass{\amscodelist}
\maketitle
\or
\maketitle
\begin{abstract}
  \abstracttext
\end{abstract}
\fi

\section{Introduction}
\label{sec:introduction}

Non-smooth nonlinear optimization problems of the form
 \begin{equation*}
   \minst[x]{f(x)}
   g(x) = 0,\quad h(x) \ge 0, \tag{NLP} \label{eq:nlp}
 \end{equation*}
where $\Domx\subseteq \R^n$ is open, the objective $f\in\Cd(\Domx,\R)$ is a smooth function
and the equality and inequality constraints $g\in\Cabs(\Domx,\R^{m_1})$ and
$h\in\Cabs(\Domx,\R^{m_2})$ are level-1 non-smooth functions
that can be written in \emph{abs-normal form} \cite{Griewank_2013}
have been considered by the authors in \cite{Hegerhorst_Steinbach:2019}.
In this problem class, the non-smoothness is caused by finitely many occurrences
of the absolute value function, the branches of which we represent by
signature matrices $\Sigma = \diag(\sigma)$ with $\sigma \in \set{-1,0,1}^s$.
We find functions
$\ce\in\Cd(\Domxz,\R^{m_1})$,
$\ci\in\Cd(\Domxz,\R^{m_2})$ and
$\cz\in\Cd(\Domxz,\R^s)$ with $\Domxz=\Domx\x\Domz$,
$\Domz\subseteq\R^s$ open and \emph{symmetric}
(i.e., $z \in \Domz$ implies $\Sigma z \in \Domz$
for every signature matrix $\Sigma$) such that
\begin{align*}
g(x)&=\ce(x,\abs{z}),\\
h(x)&=\ci(x,\abs{z}),\tag{ANF}\label{eq:anf}\\
z&=\cz(x,\abs{z}) \qtext{with $\partial_2 \cz(x,\abs{z})$ strictly lower triangular}.
\end{align*}
Here we use a single joint \emph{switching constraint} $\cz$
for both $g$ and $h$, and reuse \emph{switching variables}~$z_i$
if the same argument repeats as an absolute value argument in $g$ or $h$.
Due to the strictly lower triangular form of $\partial_2 \cz(x,\abs{z})$,
component $z_j$ of $z$ can be computed from $x$ and the components $z_i$, $i<j$.
Hence, the variable $z$ is implicitly defined by $z=\cz(x,\abs{z})$,
and to denote this dependence explicitly, we write $z(x)$ in the following.
Whenever we address questions of solvability of this system, we make use of the reformulation $\abs{z_i} = \sign(z_i) z_i$.

\begin{definition}[Signature of $z$]
  \label{def:signature}
  Let $x \in \Domx$. We define the \emph{signature} $\sigma(x)$
  and the associated \emph{signature matrix} $\Sigma(x)$ as
  \begin{equation*}
    \sigma(x) \define \sign(z(x)) \in \set{-1,0,1}^s \qtextq{and}
    \Sigma(x) \define \diag(\sigma(x)).
  \end{equation*}
  A signature vector $\sigma(x) \in \set{-1,1}^s$ is called \emph{definite},
  otherwise \emph{indefinite}.
\end{definition}

For signatures $\sigma,\^\sigma \in \set{-1, 0, 1}^s$,
it is convenient to use the partial order
\begin{equation*}
  \^\sigma \succeq \sigma \iff
  \^\sigma_i \sigma_i \ge \sigma_i^2 \qtextq{for} i=1,\dots,s,
\end{equation*}
i.e., $\^\sigma_i$ is arbitrary if $\sigma_i = 0$ and $\^\sigma_i = \sigma_i$ otherwise.
Thus, we may write $\abs{z(x)}=\Sigma z(x)$ for every $\sigma \succeq \sigma(x)$.
Further, we may consider the system $z=\cz(x,\Sigma z)$ \textbf{for fixed signature} $\Sigma = \Sigma(\^x)$ around a point of interest $\^x$.
By the implicit function theorem, the system has a locally unique solution $z(x)$ for fixed signature $\Sigma$, and the associated Jacobian at $\^x$ reads
\begin{equation*}
  \partial_x z(\^x) = [I - \partial_2 \cz(\^x,\abs{z(\^x)}) \Sigma]^{-1}
  \partial_1 \cz(\^x,\abs{z(\^x)})\in\R^{s \x n}.
\end{equation*}

\begin{definition}[Active Switching Set]
  We call the switching variable $z_i$ \emph{active} if $z_i(x) = 0$.
  The \emph{active switching set} $\alpha$
  consists of all indices of active switching variables,
  \begin{equation*}
    \alpha(x) \define \defset{1 \leq i \leq s}{z_i(x) = 0}.
  \end{equation*}
  The numbers of active and inactive switching variables are $\calp[(x)]$
  and $\csig[(x)] \define s - \calp[(x)]$.
\end{definition}

\subsection*{Literature}

Griewank and Walther have developed a class of
unconstrained abs-normal problems in \cite{Griewank_2013,Griewank_Walther_2016}.
These problems offer particularly attractive theoretical features
when generalizing KKT theory and stationarity concepts to non-smooth problems.
Under certain regularity conditions, they are computationally tractable by active-set
type algorithms with guaranteed convergence based on piecewise linearizations and using
algorithmic differentiation techniques \cite{Griewank_Walther_2017,Griewank_Walther_2019}.

Another important class of non-smooth optimization problems are Mathematical
Programs with Complementarity (or Equilibrium) Constraints (MPCCs, MPECs);
an overview can be found in the book \cite{Luo_et_al:1996}.
Since standard theory for smooth optimization problems cannot be applied, new constraint qualifications and corresponding optimality conditions were introduced.
By now, there is a large body of literature on MPCCs, and we refer to
\cite{Ye:2005} for an overview of the basic concepts and theory.
In this paper, constraint qualifications for MPCCs in the sense of Abadie and Guignard and corresponding stationarity concepts (in particular M-stationarity and MPCC-linearized B-stationarity) are considered.
Details can be found in \cite{Scheel_Scholtes_2000}, \cite{Luo_et_al:1996} and \cite{FlegelDiss}.

In \cite{Hegerhorst_et_al:2019:MPEC1} we have shown that
unconstrained abs-normal problems constitute a subclass of MPCCs.
In addition, we have studied regularity concepts of linear independence and of
Mangasarian-Fromovitz type.  
As a direct generalization of unconstrained abs-normal problems
we have considered NLPs with abs-normal constraints,
which turned out to be equivalent to the class of MPCCs.
In \cite{Hegerhorst_Steinbach:2019} we have extended optimality conditions of unconstrained abs-normal problems
to general abs-normal NLPs under the linear independence kink qualification
using a reformulation of inequalities with absolute value slacks.
We have compared these optimality conditions to concepts of MPCCs
in \cite{Hegerhorst_et_al:2019:MPEC2}.
We have also shown that the above slack reformulation preserves kink qualifications of linear independece type but not of Mangasarian-Fromovitz type.
More details and additional information about these results as well as about the results in this paper can be found in \cite{Hegerhorst-Schultchen2020}.

\subsection*{Contributions.} In the present article we extend our detailed comparative study of general abs-normal NLPs and MPECs,
considering constraint qualifications of Abadie and Guignard type both for the standard formulation
and for the reformulation with absolute value slacks.
In particular, we show that constraint qualifications of Abadie type are equivalent for
abs-normal NLPs and MPCCs and that they are preserved under the slack reformulation.
For constraint qualifications of Guignard type we cannot prove equivalence
but only certain implications.
However, when considering branch problems of abs-normal NLPs and MPCCs,
we again obtain equivalence of constraint qualifications of Abadie
and Guignard type, even under the slack reformulation.
Finally we introduce Mordukhovich and Bouligand stationarity concepts
for abs-normal NLPs and prove first order optimality conditions
using the corresponding concepts for MPCCs.

\subsection*{Structure.} The remainder of this article is structured as follows.
In \cref{sec:anf-formulations}
we state the general abs-normal NLP and its reformulation with absolute value slacks that
permits to dispose of inequalities. We also present the branch structure of both formulations
and introduce appropriate definitions of the \tcone and the linearized cone. Using these tools,
we introduce kink qualifications in the sense of Abadie and Guignard.
In terms of these two kink qualifications, we then compare the regularity of the equality-constrained form of an
abs-normal NLP to the inequality-constrained one.
In \cref{sec:counterpart} we introduce counterpart MPCCs
for the two formulations of abs-normal NLPs
and discuss the associated MPCC-constraint qualifications, namely
MPCC-ACQ and MPCC-GCQ.
In \cref{sec:qualifications} we investigate the interrelation of the regularity concepts
for abs-normal NLPs and MPECs and find the situation to be more intricate than
under LICQ and MFCQ discussed in \cite{Hegerhorst_Steinbach:2019}.
Finally, in \cref{sec:stationarity} we introduce abs-normal variants of M-stationarity and B-stationarity
as first order necessary optimality conditions for abs-normal NLPs
and prove equivalence relations for the respective MPCC stationarity conditions.
We conclude with \cref{sec:conclusions}. 

\section{Inequality and Equality Constrained Formulations}
\label{sec:anf-formulations}

In this section we consider two different treatments of inequality constraints
for non-smooth NLPs in abs-normal form.

\subsection{General Abs-Normal NLPs}
\label{sec:anf-inequalities}

Substituting the representation \eqref{eq:anf} of constraints in abs-normal form
into the general problem \eqref{eq:nlp}, we obtain a general abs-normal NLP.
Here, we use the variables $(t,\zt)$ instead of $(x,z)$ and analogously $\sigt(t)$ and $\alpt(t)$ instead of $\sigma(x)$ and $\alpha(x)$.

\begin{definition}[Abs-Normal NLP]
   Let $\Domt$ be an open subset of $\R^{n_t}$.
   A non-smooth NLP is called an \emph{abs-normal} NLP
   if functions $f \in \Cd(\Domt,\R)$, $\ce \in \Cd(\Domtzt,\R^{m_1})$,
   $\ci \in \Cd(\Domtzt,\R^{m_2})$, and $\cz \in \Cd(\Domtzt,\R^{s_t})$ with $d \ge 1$ exist
   such that it reads
   \begin{equation*}
     \minst[t, \zt]{f(t)}
      \ce(t, \abs{\zt}) = 0,\quad
      \ci(t, \abs{\zt}) \ge 0, \tag{I-NLP} \label{eq:i-anf}\quad
      \cz(t, \abs{\zt}) - \zt = 0,
   \end{equation*}
   where $\Domzt$ is open and symmetric and
   $\partial_2 \cz(x,\abs{\zt})$ is strictly lower triangular.
   The feasible set of \eqref{eq:i-anf} is
   $\Fabs \define \defset{(t, \zt)}{
     \ce(t, \abs\zt) = 0,
     \ci(t, \abs\zt) \ge 0,
     \cz(t, \abs\zt) = \zt
   }$.
\end{definition}
\begin{definition}[Active Inequality Set]
  Let $(t,\zt(t))\in \Fabs$. We call the inequality constraint
  $i\in\setI$ \emph{active} if $c_i(t,\abs{\zt(t)})=0$.
  The \emph{active set} $\setA(t)$ consists of all indices of active inequality constraints,
   $ \setA(t)=\defset{i\in\setI}{c_i(t,\abs{\zt(t)})=0}.$
  We set $\cA \define [c_i]_{i \in \setA(t)}$
  and denote the number of active inequality constraints by $\card{\setA(t)}$.
\end{definition}

With the goal of considering kink qualifications in the spirit of Abadie and Guignard,
we define the \tcone and the abs-normal-linearized cone.

\begin{definition}[\TCone and Abs-Normal-Linearized Cone for \eqref{eq:i-anf}]\label{def:cones-i-anf}
Consider a feasible point $(t,\zt)$ of \eqref{eq:i-anf}. The \emph{\tcone} to $\Fabs$ at $(t,\zt)$ is
\begin{equation*}
  \Tabs(t, \zt) \define
  \begin{defarray}{(\delta t, \delta\zt)}
    \exists \tau_k \searrow 0,\ \Fabs \ni (t_k,\zt_k) \to (t, \zt){:} \\
    \tau_k^{-1} (t_k - t, \zt_k - \zt) \to (\delta t, \delta\zt)
  \end{defarray}
  .
\end{equation*}
With $\delta\zeta_i:=\abs{\delta\zt_i}$ if $i\in\alpt(t)$
and $\delta\zeta_i:=\sigt_i(t) \delta\zt_i$ if $i\notin\alpt(t)$,
the \emph{abs-normal-linearized cone} is
\begin{equation*}
  \Tlinabs(t, \zt) \define
  \begin{defarray}[r@{\medspace}l]{(\delta t, \delta\zt)}
    \partial_1 \ce(t, \abs\zt) \delta t +
    \partial_2 \ce(t, \abs\zt) \delta\zeta &= 0, \\
    \partial_1 \cA(t, \abs\zt) \delta t +
    \partial_2 \cA(t, \abs\zt) \delta\zeta &\ge 0, \\
    \partial_1 \cz(t, \abs\zt) \delta t +
    \partial_2 \cz(t, \abs\zt) \delta\zeta &= \delta\zt
  \end{defarray}
  .
\end{equation*}
\end{definition}

To prove that the \tcone is a subset of the abs-normal-linearized cone,
we follow an idea from \cite{FlegelDiss}, where an analogous result for MPCCs was obtained.
First, we need the definition of the smooth branch NLPs for \eqref{eq:i-anf}
with their standard \tcones and linearized cones.

\begin{definition}[Branch NLPs for \eqref{eq:i-anf}]\label{def:branch-anf}
 Consider a feasible point $(\^t,\hzt)$ of \eqref{eq:i-anf}.
 Choose $\sigt \in\{-1,1\}^{s_t}$ with $\sigt \succeq \sigt(\^t)$ and set $\Sigt=\diag(\sigt)$.
 The branch problem NLP($\Sigt$) is defined as
  \ifcase0
  \begin{align*}
    \sminst[t,\zt]{f(t)}
    & \ce(t,\Sigt \zt) = 0, \quad
    \ci(t,\Sigt \zt) \ge 0, \\
    & \cz(t,\Sigt \zt) - \zt = 0, \quad
    \Sigt \zt\ge 0.
    \tag{NLP($\Sigt$)}\label{eq:branch-anf}
  \end{align*}
  \else
  \begin{align*}
    \minst[t,\zt]{f(t)}
    & \ce(t,\Sigt \zt) = 0,
     \ci(t,\Sigt \zt) \ge 0,\tag{NLP($\Sigt$)}\label{eq:branch-anf}
     \cz(t,\Sigt \zt) = \zt,
     \Sigt \zt\ge 0.
  \end{align*}
  \fi
  The feasible set of \eqref{eq:branch-anf},
  which always contains $(\^t,\hzt)$, is denoted by $\Fsig$.
\end{definition}

\begin{definition}[\TCone and Linearized Cone for \eqref{eq:branch-anf}]\label{def:cones-branch-anf}
 Given \eqref{eq:branch-anf}, consider a feasible point $(t,\zt)$.
 The \emph{\tcone} to $\Fsig$ at $(t,\zt)$ is
  \begin{equation*}
    \Tsig(t,\zt) \define
    \begin{defarray}{(\delta t,\delta\zt)}
      \exists \tau_k \searrow 0,\ \Fsig \ni (t_k, \zt_k) \to (t, \zt){:} \\
      \tau_k^{-1} (t_k - t, \zt_k - \zt) \to (\delta t, \delta\zt)
    \end{defarray}
    .
  \end{equation*}
  The \emph{linearized cone} is
  \begin{equation*}
    \Tlinsig(t, \zt) \define
    \begin{defarray}[r@{\medspace}l]{(\delta t,\delta\zt)}
      \partial_1 \ce(t, \Sigt\zt) \delta t +
      \partial_2 \ce(t, \Sigt\zt) \Sigt\delta\zt &= 0, \\
      \partial_1 \cA(t, \Sigt\zt) \delta t +
      \partial_2 \cA(t, \Sigt\zt) \Sigt\delta\zt &\ge 0, \\
      \partial_1 \cz(t, \Sigt\zt) \delta t +
      \partial_2 \cz(t, \Sigt\zt) \Sigt\delta\zt &= \delta\zt, \\
      \sigt_i \delta\zt_i & \ge 0,\ i \in \alpt(t)
    \end{defarray}
    .
  \end{equation*}
\end{definition}

\begin{remark}
  Observe that $\abs{\zt} = \Sigt \zt$ in \cref{def:branch-anf} and \cref{def:cones-branch-anf},
  and for every $\Sigt$ we have $\Fsig \subseteq \Fabs$,
  $\Tsig(t,\zt) \subseteq \Tabs(t,\zt)$, and
  $\Tlinsig(t,\zt) \subseteq \Tlinabs(t,\zt)$.
\end{remark}

\begin{lemma}\label{le:decomp-cones-anf}
Consider a feasible point $(\^t,\hzt)$ of \eqref{eq:i-anf} with associated branch problems \eqref{eq:branch-anf}.
Then, the following decompositions of the \tcone and of the abs-normal-linearized cone of \eqref{eq:i-anf} hold:
 \begin{equation*}
  \Tabs(\^t,\hzt)=\bigcup_{\Sigt} \Tsig(\^t,\hzt)
  \qtextq{and}
  \Tlinabs(\^t,\hzt)=\bigcup_{\Sigt} \Tlinsig(\^t,\hzt).
 \end{equation*}
\end{lemma}
\begin{proof}
 We first consider the \tcones and show that a neighborhood $\setN$ of $(\^t,\hzt)$ exists such that
 \begin{equation*}
  \Fabs\cap \setN = \bigcup_{\Sigt} (\Fsig\cap \setN).
 \end{equation*}
 The inclusion $\supseteq$ holds for every neighborhood $\setN$
 since $\Fsig\subseteq\Fabs$ for all $\Sigt$.
 To show the inclusion $\subseteq$ we consider an index $i\notin\alpt(\^t)$.
 Then, by continuity, $\epsilon_i>0$ exists with
 $\sigt_i(t) = \sigt_i(\^t) \in \set{-1,+1}$ for all $t\in B_{\epsilon_i}(\^t)$.
 Now set $\epsilon\define\min_{i\notin\alpt(\^t)}\epsilon_i$,
 $\setN\define B_{\epsilon} \x \R^{n_t}$,
 and consider $(t, \zt) \in \setN \cap \Fabs$.
 With the choice $\sigt_i=\sigt_i(t)$ for $i\notin\alpt(t)$
 and $\sigt_i=1$ for $i\in\alpt(t)$
 we find $\Sigt=\diag(\sigt)$ such that
 $(t, \zt) \in \setN\cap \Fsig$ since $\alpt(t) \subseteq \alpt(\^t)$. Thus,
 \begin{equation*}
  \Fabs\cap \setN = \bigcup_{\Sigt} (\Fsig\cap \setN).
 \end{equation*}
 Now, let $\setT(\^t,\hzt; \setF)$ generically denote
 the \tcone to $\setF$ at $(\^t, \hzt)$.
 Then,
 \begin{align*}
   \Tabs(\^t,\hzt)
   =
   \setT(\^t,\hzt; \Fabs)
   &=
   \setT(\^t,\hzt; \Fabs\cap \setN)
   =
   \setT(\^t,\hzt; {\textstyle\bigcup_{\Sigt}} (\Fsig \cap \setN)) \\
   &=
   \bigcup_{\Sigt} \setT(\^t,\hzt; \Fsig \cap \setN)
   =
   \bigcup_{\Sigt} \setT(\^t,\hzt; \Fsig)
   =
   \bigcup_{\Sigt} \Tsig(\^t,\hzt).
 \end{align*}
 Here the fourth equality holds since the number of branch problems is finite.
 The decomposition of $\Tlinabs$ follows directly
 by comparing definitions of $\Tlinabs$ and $\Tlinsig$.
\end{proof}

\begin{lemma}\label{le:i-cones}
Let $(t,\zt)$ be feasible for \eqref{eq:i-anf}. Then,
 \begin{equation*}
  \Tabs(t,\zt)\subseteq \Tlinabs(t,\zt) \qtextq{and} \Tabs(t,\zt)^* \supseteq \Tlinabs(t,\zt)^*.
 \end{equation*}
\end{lemma}
\begin{proof}
 The branch NLPs are smooth, hence the inclusion $\Tsig(t,\zt) \subseteq \Tlinsig(t,\zt)$ holds by standard NLP theory.
 Then, the first inclusion follows directly from \cref{le:decomp-cones-anf}
 and the second inclusion follows by dualization of the cones.
\end{proof}

In general, the reverse inclusions do not hold. This leads to the following definitions.

\ifCmp
\begin{definition}
  [Abadie's and Guignard's Kink Qualifications for \eqref{eq:i-anf}]
  \label{def:akq}
  \label{def:gkq}
  Consider a feasible point $(t,\zt(t))$ of \eqref{eq:i-anf}.
  We say that \emph{Abadie's Kink Qualification (AKQ)} holds at $t$
  if \ifnum\FmtChoice=2 we have \fi $\Tabs(t,\zt(t)) = \Tlinabs(t,\zt(t))$,
  and that \emph{Guignard's Kink Qualification (GKQ)} holds at $t$
  if $\Tabs(t,\zt(t))^* = \Tlinabs(t,\zt(t))^*$.
\end{definition}
\fi

The decomposition of cones in \cref{le:decomp-cones-anf}
and its dualization immediately lead to the next
\ifCmp theorem\fi.

\ifCmp
\begin{theorem}[ACQ/GCQ for all \eqref{eq:branch-anf}
  implies AKQ/GKQ for \eqref{eq:i-anf}]
  \label{th:branch-acq_akq}
  \label{th:branch-gcq_gkq}
  Consider a feasible point $(t,\zt(t))$ of \eqref{eq:i-anf}
  with associated branch problems \eqref{eq:branch-anf}.
  Then, AKQ respectively GKQ holds for \eqref{eq:i-anf} at $t$
  if ACQ respectively GCQ holds for all \eqref{eq:branch-anf} at $(t,\zt(t))$.
\end{theorem}
\fi

\subsection{Abs-Normal NLPs with Inequality Slacks}
\label{sec:anf-equalities}

Here, we use absolute values of slack variables to get rid of the inequality constraints.
This idea is due to Griewank. It has been introduced in \cite{Hegerhorst_Steinbach:2019}
and has been further investigated in \cite{Hegerhorst_et_al:2019:MPEC2}.
With slack variables $w\in \R^{m_2}$, we reformulate \eqref{eq:nlp} as follows:
   \begin{equation*}
     \minst[t,w]{f(t)}
      g(t) = 0,\quad
      h(t)-\abs{w} = 0.
   \end{equation*}
Then, we express $g$ and $h$ in abs-normal form as in \eqref{eq:anf}
and introduce additional switching variables $\zw$ to handle $\abs{w}$.
We obtain a class of purely equality-constrained abs-normal NLPs.
\begin{definition}[Abs-Normal NLP with Inequality Slacks]
  An abs-normal NLP posed in the following form is called an \emph{abs-normal NLP with inequality slacks}:
  \begin{align*}
    \sminst[t,w,\zt,\zw]{f(t)}
    & \ce(t,\abs{\zt}) = 0, \quad
     \ci(t,\abs{\zt})-\abs{\zw} = 0,\\ \tag{E-NLP}\label{eq:e-anf}
    & \cz(t,\abs{\zt})=\zt, \quad
     w=\zw,
  \end{align*}
  where $\Domzt$ is open and symmetric and
  $\partial_2 \cz(t,\abs{\zt})$ is strictly lower triangular.
  The feasible set of \eqref{eq:e-anf} is denoted by $\Feabs$ and is a lifting of $\Fabs$.
\end{definition}

 \begin{remark}
   Introducing $\abs{w}$ converts inequalities
   to pure equalities without a nonnegativity condition
   for the slack variables $w$.
   In \cite{Hegerhorst_Steinbach:2019} we have used this formulation to simplify
   the presentation of first and second order conditions for the general
   abs-normal NLP
   under the linear independence kink qualification (LIKQ).
   Later we will see that constraint qualifications of Abadie type are preserved under reformulation.
   Nevertheless, this representation causes some problems.
   In \cite{Hegerhorst_et_al:2019:MPEC2} we have shown that,
   in contrast to LIKQ,
   constraint qualifications of Mangasarian-Fromovitz type are not preserved.
   Moreover, we cannot prove compatibility of constraint qualifications of Guignard type.
   Also, note that the equation $w - \zw = 0$ (and hence $w$) cannot be eliminated as
   this would destroy the abs-normal form.
   Finally, the signs of nonzero components $w_i$ can be chosen arbitrarily and thus the slack $w$ is not uniquely determined.
   This needs to be taken into account when formulationg
   kink qualifications (KQ) for \eqref{eq:e-anf}.
 \end{remark}

We are now interested in \ifCmp deriving \fi
Abadie's and Guignard's KQ for \eqref{eq:e-anf}.
To this end, we observe that the
formulation \eqref{eq:e-anf} can be seen as a special case of \eqref{eq:i-anf}:
Let $x=(t,w)$, $z=(\zt,\zw)$, $\_f(x)=f(t)$,
$\bcE(x,\abs{z})=(\ce(t,\abs{\zt}), \ci(t,\abs{\zt})-\abs{\zw})$, and
$\bcZ(x,\abs{z})=(\cz(t,\abs{\zt}),w)$.
Then, we can rewrite \eqref{eq:e-anf} as
  \begin{equation*}
   \minst[x,z]{f(x)}
    \bcE(x,\abs{z}) = 0,\quad
    \bcZ(x,\abs{z})-z = 0.
  \end{equation*}
\ifCmp
Hence, the following material is readily obtained by specializing
the definitions and results in the previous section.

\else
Hence, the following lemmas follow directly from results in the previous section.
\begin{lemma}[\TCone and Abs-Normal-Linearized Cone for \eqref{eq:e-anf}]\label{le:cones-e-anf}
\fi
With $\delta=(\delta t, \delta w, \delta\zt,\delta \zw)$,
\ifCmp \cref{def:cones-i-anf} and $w = \zw$ give \fi
the \tcone to $\Feabs$ at $(t, w,\zt, \zw)$ \ifCmp as \fi
\begin{equation*}
  \Teabs(t, w, \zt, \zw) =
  \begin{defarray}{\delta}
      \exists \tau_k \searrow 0, \
      \Feabs \ni (t_k, w_k, \zt_k, \zw_k) \to (t, w, \zt, \zw){:} \\
      \tau_k^{-1} (t_k - t, w_k - w, \zt_k - \zt)
      \to (\delta t, \delta w, \delta\zt), \
      \delta \zw = \delta w
    \end{defarray}
    ,
  \end{equation*}
and the abs-normal-linearized cone reads
  \begin{equation*}
    \Tlineabs(t, w, \zt, \zw) =
    \ifcase0
    \begin{defarray}{\delta}
      \partial_1 \ci(t, \abs\zt) \delta t +
      \partial_2 \ci(t, \abs\zt) \delta\zeta = \delta\omega, \\
      (\delta t, \delta\zt) \in \Tlinabs(t,\zt), \
      \delta\zw = \delta w
    \end{defarray},
    \or
    \begin{defarray}[r@{\medspace}l]{\delta}
      (\delta t, \delta\zt)&\in\Tlinabs(t,\zt), \\
      \partial_1 \ci(t, \abs\zt) \delta t +
      \partial_2 \ci(t, \abs\zt) \delta\zeta &= \delta\omega, \\
      \delta w&=\delta\zw
    \end{defarray},
    \fi
  \end{equation*}
where $\alpha=(\alpt,\alpw)$ and
\begin{equation*}
  \delta\zeta_i
  =
  \begin{ccases}
    \sigt_i(t) \delta\zt_i, & i\notin\alpt(t) \\
    \abs{\delta\zt_i}, & i\in\alpt(t)
  \end{ccases}
  ,
  \quad
  \delta\omega_i
  =
  \begin{ccases}
    \sigw_i(w) \delta\zw_i, & i\notin\alpw(w) \\
    \abs{\delta\zw_i}, & i\in\alpw(w)
  \end{ccases}
  .
\end{equation*}
\ifCmp
In \cref{def:branch-anf}, consider
\fi
 a feasible point $(\^t,\^w,\hzt,\hzw)$ of \eqref{eq:e-anf}.
 Choose $\sigt\in\{-1,1\}^{s_t}$ with $\sigt \succeq \sigt(\^t)$ and
 $\sigw\in\{-1,1\}^{m_2}$ with $\sigw \succeq \sigw(\^w)$.
 Set $\Sigt=\diag(\sigt)$ and $\Sigw=\diag(\sigw)$.
 Then, the branch problem NLP($\Sigtw$) for $\Sigtw \define \diag(\Sigt, \Sigw)$ reads
  \begin{align*}
    \sminst[t,w,\zt,\zw]{f(t)}
    & \ce(t,\Sigt \zt) = 0, \quad
     \ci(t,\Sigt \zt) - \Sigw \zw = 0,  \\
    & \cz(t,\Sigt \zt) - \zt = 0, \quad
     w - \zw = 0,\tag{NLP($\Sigtw$)}\label{eq:branch-e-anf}\\
     & \Sigt \zt \ge 0, \quad
     \Sigw \zw \ge 0.
  \end{align*}
  The feasible set of \eqref{eq:branch-e-anf},
  which always contains $(\^t,\^w,\hzt,\hzw)$, is denoted by $\Fesig$ and is a lifting of $\Fsig$.
\ifCmp
By \cref{def:cones-branch-anf}, the
\fi
 \tcone to $\Fesig$ at $(t,w,\zt,\zw)$ reads
  \begin{equation*}
    \Tesig(t,w,\zt,\zw) =
    \begin{defarray}{\delta}
      \exists \tau_k \searrow 0,\ \Fesig \ni (t_k,w_k,\zt_k,\zw_k) \to (t,w,\zt,\zw){:} \\
      \tau_k^{-1} (t_k - t, w_k-w, \zt_k - \zt) \to (\delta t, \delta w, \delta\zt), \
      \delta \zw = \delta w
    \end{defarray}
  \end{equation*}
  with $\delta=(\delta t, \delta w, \delta\zt, \delta \zw)$.
  The linearized cone reads
  \begin{equation*}
    \Tlinesig(t,w,\zt,\zw) =
    \ifcase0
    \begin{defarray}{\delta}
      \partial_1 \ci \delta t +
      \partial_2 \ci \Sigt \delta\zt - \Sigw \delta\zw = 0, \
      \delta\zw = \delta w, \\
      (\delta t,\delta\zt) \in \Tlinsig(t,\zt), \
      \sigw_i \delta\zw_i \ge 0, \ i \in  \alpw(w)
    \end{defarray}
    \or
    \begin{defarray}[r@{\medspace}l]{\delta}
      \partial_1 \ci \delta t +
      \partial_2 \ci \Sigt \delta\zt - \Sigw \delta\zw &= 0, \\
      (\delta t,\delta\zt) \in \Tlinsig(t,\zt), \
      \delta\zw &= \delta w, \\
      \sigw_i \delta\zw_i \ge 0,\ i & \in  \alpw(w)
    \end{defarray}
    \or
    \begin{defarray}[r@{\medspace}l]{\delta}
      (\delta t,\delta\zt) &\in \Tlinsig(t,\zt), \\
      \partial_1 \ci \delta t +
      \partial_2 \ci \Sigt \delta\zt - \Sigw \delta\zw &= 0, \\
      \delta w &= \delta \zw, \\
      \sigw_i \delta\zw_i \ge 0,\ i & \in  \alpw(w)
    \end{defarray}
    \fi
    .
  \end{equation*}
  Here, all partial derivatives are evaluated at $(t, \Sigt\zt)$.
\ifCmp\else
\end{lemma}
\begin{proof}
 This follows from \cref{def:cones-branch-anf}.
\end{proof}
\fi

Moreover, we obtain the following decompositions by applying
\cref{le:decomp-cones-anf} to \eqref{eq:e-anf}
at $y=(t,w,\zt,\zw)$ with associated branch problems $\eqref{eq:branch-e-anf}$:
\begin{equation*}
  \Teabs(y) = \bigcup_{\Sigtw} \Tesig(y) \qtextq{and} \Tlineabs(y) = \bigcup_{\Sigtw} \Tlinesig(y).
\end{equation*}
As before, the \tcone is a subset of the linearized cone
and the reverse inclusion holds for the dual cones:
\begin{equation*}
  \Teabs(y)
  \subseteq
  \Tlineabs(y)
   \qtextq{and}
  \Teabs(y)^*
  \supseteq
  \Tlineabs(y)^*.
\end{equation*}
This follows directly by applying \cref{le:i-cones} to \eqref{eq:e-anf}.
Again, equality does not hold in general,
and we consider Abadie's Kink Qualification (AKQ)
and Guignard's Kink Qualification (GKQ) for \eqref{eq:e-anf}.

\ifCmp
Given a feasible point $y = (t,w,\zt(t),\zw(w))$ of \eqref{eq:e-anf},
\cref{def:akq} gives AKQ and GKQ at $(t,w)$, respectively, as
\begin{equation*}
  \Teabs(y) = \Tlineabs(y)
  \qtextq{and}
  \Teabs(y)^* = \Tlineabs(y)^*.
\end{equation*}
\fi

\begin{remark}\label{rem:independece}
The possible slack values
$w \in W(t) \define \defset{w}{\abs{w} = \ci(t,\abs{\zt(t)})}$
just differ by the signs of components $w_i$ for $i \in \setA(t)$.
Thus, neither AKQ nor GKQ depends on the particular choice of $w$,
and both conditions are well-defined for $\eqref{eq:e-anf}$.
\end{remark}

\ifCmp
Now \cref{th:branch-acq_akq} takes the following form.
\begin{theorem}
  [ACQ/GCQ for all \eqref{eq:branch-e-anf} implies AKQ/GKQ for \eqref{eq:e-anf}]
  \label{th:e-branch-acq_akq}
  \label{th:e-branch-gcq_gkq}
  Consider a feasible point $y = (t,w,\zt(t),\zw(w))$ of \eqref{eq:e-anf}
  with associated branch problems \eqref{eq:branch-e-anf}.
  Then, AKQ respectively GKQ for \eqref{eq:e-anf} holds at $(t,w)$ if
  ACQ respectively GCQ holds for all \eqref{eq:branch-e-anf} at $y$.
\end{theorem}
\fi

\subsection{Relations of Kink Qualifications for Abs-Normal NLPs}

In this paragraph we discuss the relations of kink qualifications for the two different formulations introduced above.
Here, equality of the cones and of the dual cones just needs to be considered for one element of the set $W(t) = \defset{w}{\abs{w}=\ci(t,\abs{\zt(t)})}$.
Then, it holds directly for all other elements by \cref{rem:independece}.

\begin{theorem}\label{th:akq}
 AKQ for \eqref{eq:i-anf} holds at $(t,\zt(t))\in\Fabs$ if and only if
 AKQ for \eqref{eq:e-anf} holds at $(t,w,\zt(t),\zw(w))\in\Feabs$
 for any (and hence all) $w\in W(t)$.
\end{theorem}
\begin{proof}
 As $\Tabs(t,\zt)\subseteq \Tlinabs(t,\zt)$ and
 $\Teabs(t,\zt)\subseteq \Tlineabs(t,\zt)$
 always hold, we just need to prove
 \begin{equation*}
   \Tabs(t, \zt) \supseteq \Tlinabs(t, \zt)
   \iff
   \Teabs(t, w, \zt, \zw) \supseteq \Tlineabs(t, w, \zt, \zw).
 \end{equation*}
 We start with the implication ``$\Rightarrow$''. Let $\delta=(\delta t,\delta w,\delta \zt, \delta\zw)\in \Tlineabs(t,w, \zt, \zw)$.
 Then, we have $\tilde\delta=(\delta t,\delta \zt)\in \Tlinabs(t,\zt)=\Tabs(t,\zt)$.
 Hence, there exist sequences $(t_k,\zt_k)\in\Fabs$ and $\tau_k\searrow 0$ with $(t_k,\zt_k)\to(t,\zt)$ and $\tau_k^{-1}(t_k-t, \zt_k-\zt)\to (\delta t,\delta\zt)$.
 Now, define
 \begin{equation*}
  \Sigw=\diag(\sigma) \qtextq{with} \sigma_i=\begin{cases} \sigw(w_i), & i\notin\alpw(w), \\ \sign(\delta\zw_i), & i\in\alpw(w), \end{cases}
 \end{equation*}
 and set $\zw_k\define w_k\define\Sigw\ci(t_k,\abs{\zt_k})$.
 Then, we have $\zw=w= \Sigw\ci(t,\abs{\zt})$ and obtain
 \begin{align*}
  \zw_k-\zw
  &=
  \Sigw [\ci(t_k,\abs{\zt_k})-\ci(t,\abs{\zt})] \\
  &=
  \Sigw [
    \partial_1 \ci(t,\abs{\zt})(t_k-t) +
    \partial_2 \ci(t,\abs{\zt})(\abs{\zt_k}-\abs{\zt}) +
    o(\norm{(t_k-t, \abs{\zt_k}-\abs{\zt})})
  ].
 \end{align*}
 Further, for $k$ large enough we have
 $\abs{\zt_k}-\abs{\zt}=\Sigt_k\zt_k-\Sigt\zt$ using
 $\Sigt_k=\diag(\sigt_k)$ with $\sigt_k=\sigma(t_k)$ and $\Sigt=\diag(\sigt)$ with $\sigt=\sigma(t)$.
 Then, we obtain for $\zt_i\neq 0$
 \begin{equation*}
  \tau_k^{-1}(\abs{(\zt_k)_i}-\abs{\zt_i}) = \tau_k^{-1}\sigt_i ((\zt_k)_i-\zt_i)\to \sigt_i\delta\zt_i.
 \end{equation*}
 For $\zt_i= 0$ we have $\tau_k^{-1}(\zt_k)_i \to \delta\zt_i$ and hence
 \begin{equation*}
  \tau_k^{-1}(\abs{(\zt_k)_i}-\abs{\zt_i}) =
  \tau_k^{-1}\abs{(\zt_k)_i} \to \abs{\delta\zt_i}.
 \end{equation*}
 Thus,
 $ \tau_k^{-1}(\abs{(\zt_k)}-\abs{\zt}) \to \delta\zeta$
 holds, and in total
 \begin{equation*}
  \tau_k^{-1}(\zw_k-\zw)\to \Sigw [\partial_1 \ci(t,\abs{\zt})\delta t + \partial_2 \ci(t,\abs{\zt})\delta \zeta]=\Sigw\delta\zeta=\delta\zw.
 \end{equation*}
 Additionally, we obtain $\tau_k^{-1}(w_k-w)\to \delta w$ and finally $d\in \Teabs(t,w, \zt, \zw)$.
 To prove the implication ``$\Leftarrow$'',
 consider $\delta=(\delta t,\delta \zt)\in \Tlinabs(t,\zt)$. We define
 \begin{equation*}
  \Sigw=\diag(\sigma) \qtextq{with} \sigma_i=\begin{cases} \pm 1, & i\in\setA(t), \\ \sign([\partial_1\ci(t,\abs{\zt})\delta t+\partial_2\ci(t,\abs{\zt})\delta\zeta]_i), & i\notin\setA(t), \end{cases}
 \end{equation*}
 and set $\delta w=\delta\zw=\Sigw [\partial_1\ci(t,\abs{\zt})\delta t+\partial_2\ci(t,\abs{\zt})\delta\zeta]$.
 Then\ifnum\FmtChoice=0,\else\ we have \fi $\tilde\delta=(\delta t,\delta w,\delta\zt,\delta\zw)\in\Tlineabs(t,w,\zt,\zw)$ for $w=\zw=\Sigw\ci(t,\abs{\zt})$.
 By assumption, $\tilde\delta\in\Teabs(t,w,\zt,\zw)$ holds,
 and this directly implies $\delta=(\delta t,\delta\zt)\in\Tabs(t,\zt)$.
 \end{proof}

\begin{theorem}\label{th:gkq}
 GKQ for \eqref{eq:i-anf} holds at \ifnum\FmtChoice=1 the point \fi $(t,\zt(t))\in\Fabs$ if GKQ for \eqref{eq:e-anf} holds at $(t,w,\zt(t),\zw(w))\in\Feabs$ for any (and hence all) $w\in W(t)$.
\end{theorem}
\begin{proof}
 The inclusion $\Tabs(t, \zt)^* \supseteq \Tlinabs(t, \zt)^*$ is always satisfied.
 Thus, we just have to show
\begin{equation*}
  \Tabs(t, \zt)^* \subseteq \Tlinabs(t, \zt)^*.
\end{equation*}
 Let $\omega=(\omega t,\omega \zt)\in \Tabs(t, \zt)^*$,
 i.e. $\omega^T\delta\ge 0$ for all $\delta=(\delta t,\delta \zt)\in \Tabs(t,\zt)$.
 Then, set $\tilde{\omega}=(\omega t,0,\omega \zt,0)$ and obtain $\tilde{\omega}^T\tilde \delta=\omega^T\delta\ge 0$ for all $\tilde \delta\in \Teabs(t,w,\zt,\zw)$ where $w\in W(t)$ is arbitrary.
 By assumption, then $\tilde{\omega}^T\tilde \delta\ge 0$ for all $\tilde \delta\in \Tlineabs(t,w,\zt,\zw)$ holds.
 This implies $\omega^T\delta=\tilde{\omega}^T\tilde \delta\ge 0$ for all $\delta\in \Tlinabs(t,\zt)$.
\end{proof}

The converse is unlikely to hold, but we are, at the same time, not aware of a counterexample.
Next, we consider the branch problems and relations of ACQ and GCQ for all branch problems.
Here, we can exploit sign information to show equivalence of GCQ for the
branch problems of \eqref{eq:i-anf} and \eqref{eq:e-anf}.

\begin{theorem}\label{th:branch-acq}
 ACQ for \eqref{eq:branch-anf} holds at $(t,\zt(t))\in\Fsig$ if and only if
 ACQ for \eqref{eq:branch-e-anf} holds at $(t,w,\zt(t),\zt(w))\in\Fesig$ for any (and hence all) $w\in W(t)$.
\end{theorem}
\begin{proof}
  The proof proceeds as in \cref{th:akq}.
\end{proof}

\begin{theorem}\label{th:branch-gcq}
 GCQ for \eqref{eq:branch-anf} holds at $(t,\zt(t))\in\Fsig$ if and only if
 GCQ for \eqref{eq:branch-e-anf} holds at $(t,w,\zt(t),\zt(w))\in\Fesig$ for any (and hence all) $w\in W(t)$.
\end{theorem}
\begin{proof}
The inclusions $\Tsig(t,\zt)^*\supseteq \Tlinsig(t,\zt)^*$ and
 $\Tesig(t,\zt)^*\supseteq \Tlinesig(t,\zt)^*$
 are always satisfied. Thus, we just need to prove
 \begin{equation*}
   \Tsig(t, \zt)^* \subseteq \Tlinsig(t, \zt)^*
   \iff
   \Tesig(t, w, \zt, \zw)^* \subseteq \Tlinesig(t, w, \zt, \zw)^*.
 \end{equation*}
 We start with the implication ``$\Rightarrow$''.
 Let $\omega=(\omega t,\omega w,\omega \zt,\omega\zw)\in\Tesig(t,w,\zt,\zw)^*$, i.e. $\omega^T\delta\ge 0$ for all $\delta=(\delta t,\delta w,\delta \zt\delta \zw)\in \Tesig(t,w,\zt,\zw)$.
 \ifcase0
 Set
 \begin{equation*}
   \tilde\omega
   =
   (\tilde\omega t,\tilde\omega \zt)
   =
   (\omega t,\omega \zt) +
   (\omega w + \omega \zw) \Sigw
   (\partial_1\ci(t,\Sigt \zt),
   \partial_2\ci(t,\Sigt\zt)\Sigt).
 \end{equation*}
 \or
 Set $\tilde\omega =(\tilde\omega t,\tilde\omega \zt)$ with
 \begin{align*}
   \tilde\omega t
   &\define
   \omega t + (\omega w + \omega \zw)\Sigw \partial_1\ci(t,\Sigt \zt),
   \\
   \tilde\omega\zt
   &\define
   \omega\zt +
   (\omega w + \omega \zw)\Sigw \partial_2\ci(t,\Sigt\zt)\Sigt.
 \end{align*}
 \fi
 Then, we have $\tilde\omega^T \tilde\delta = \omega^T \delta \ge 0$ for all $\delta=(\delta t,\delta \zt)\in \Tsig(t,\zt)$ and thus $\tilde\omega\in\Tlinsig(t,\zt)$.
 Then, $\omega^T\delta\ge 0$ for all $\delta=(\delta t,\delta w,\delta \zt\delta \zw)\in \Tlinesig(t,w,\zt,\zw)$ as $\omega^T\delta=\tilde\omega^T\tilde\delta$ holds.
The reverse implication may be proved as shown in \cref{th:gkq}.
\end{proof}

\section{Counterpart MPCCs}
\label{sec:counterpart}

In this section we restate the MPCC counterpart problems for the two formulations
\eqref{eq:i-anf} and \eqref{eq:e-anf} and we present the relations between them.

\subsection{Counterpart MPCC for the General Abs-Normal NLP}

To reformulate \eqref{eq:i-anf} as an MPCC, we split $\zt$ into its nonnegative part and the modulus of its
nonpositive part, $\ut \define [\zt]^+\define\max(\zt,0)$ and $\vt \define[\zt]^-\define \max(-\zt,0)$.
Then, we add complementarity of these two variables
to replace $\abs{\zt}$ by $\ut+\vt$ and $\zt$ itself by $\ut-\vt$.

\begin{definition}[Counterpart MPCC of \eqref{eq:i-anf}]
  \label{def:i-mpec}
  The \emph{counterpart MPCC} of the non-smooth NLP \eqref{eq:i-anf} reads
   \begin{align*}
     \sminst[t,\ut,\vt]{f(t)}
     & \ce(t,\ut+\vt)=0, \quad
      \ci(t,\ut+\vt)\ge 0,\\
     & \cz(t,\ut+\vt)-(\ut-\vt)=0,\tag{I-MPCC} \label{eq:i-mpec}\\
     & 0 \le \ut \compl \vt \ge 0,
   \end{align*}
   where $\ut, \vt\in\R^{s_t}$.
   The feasible set of \eqref{eq:i-mpec} is denoted by $\Fmpec$.
\end{definition}

Given an abs-normal NLP \eqref{eq:i-anf} and its counterpart MPCC \eqref{eq:i-mpec},
the mapping $\phi\: \Fmpec \to \Fabs$ defined by
\begin{equation*}
  \phi(t, \ut, \vt) = (t, \ut - \vt) \qtextq{and} \Inv\phi(t, \zt) = (t, [\zt]^+, [\zt]^-)
\end{equation*}
is a homeomorphism. This result was obtained in \cite[Lemma 31]{Hegerhorst_et_al:2019:MPEC2}.

Corresponding to the active switching set in the previous section, we introduce index sets for MPCCs.

 \begin{definition}[Index Sets]
  We denote by $\setUt_0\define\defset{i\in \set{1, \dots, s_t}}{\ut_i=0}$
  the set of indices of active inequalities $\ut_i\geq 0$,
  and by $\setUt_+\define\defset{i\in\set{1, \dots, s_t}}{\ut_i>0}$
  the set of indices of inactive inequalities $\ut_i\geq 0$.
  Analogous definitions hold of $\setVt_0$ and $\setVt_+$.
  By $\setDt\define\setUt_0\cap\setVt_0$ we denote
  the set of indices of non-strict (degenerate) complementarity pairs.
  Thus we have the partitioning
  $\set{1, \dots, s_t} = \setUt_+ \cup \setVt_+ \cup \setDt$.
\end{definition}

In order to define MPCC-CQs in the spirit of Abadie and Guignard,
we introduce the \tcone, the complementarity cone,
and the MPCC-linearized cone.

\begin{definition}[\TCone and MPCC-Linearized Cone for \eqref{eq:i-mpec}, see \cite{FlegelDiss}]\label{def:cones-i-mpec}
Consider a feasible point $(t,\ut,\vt)$ of \eqref{eq:i-mpec} with associated index sets $\setUt_+$, $\setVt_+$ and $\setDt$.
The \emph{\tcone} to $\Fmpec$ at $(t,\ut,\vt)$ is
\begin{equation*}
  \Tmpec(t, \ut, \vt) \define
  \begin{defarray}{(\delta t, \delta\ut, \delta\vt)}
    \exists \tau_k \searrow 0,\
    \Fmpec \ni (t_k, \ut_k, \vt_k) \to (t, \ut, \vt){:} \\
    \tau_k^{-1} (t_k - t, \ut_k - \ut, \vt_k - \vt)
    \to (\delta t, \delta \ut, \delta \vt)
  \end{defarray}
  .
\end{equation*}
The \emph{MPCC-linearized cone} at $(t,\ut,\vt)$ is
\begin{equation*}
  \Tlinmpec(t, \ut, \vt) \define
\ifcase1
  \begin{defarray}[r@{\medspace}l]
    {\begin{pmatrix}\delta t \\ \delta\ut \\ \delta\vt \end{pmatrix}}
    \partial_1 \ce \delta t + \partial_2 \ce (\delta\ut + \delta\vt) &= 0,\\
    \partial_1 \cA \delta t + \partial_2 \cA (\delta\ut + \delta\vt) &\ge 0,\\
    \partial_1 \cz \delta t + \partial_2 \cz (\delta\ut + \delta\vt) &=
    \delta \ut - \delta \vt, \\
    \delta \ut_i &= 0,\ i \in \setVt_+, \\
    \delta \vt_i &= 0,\ i \in \setUt_+, \\
    0 \le \delta \ut_i \compl \delta \vt_i &\ge 0,\ i \in \setDt
  \end{defarray}
\or
  \begin{defarray}[r@{\medspace}l]
    {\begin{pmatrix}\delta t \\ \delta\ut \\ \delta\vt \end{pmatrix}}
    \partial_1 \ce \delta t + \partial_2 \ce (\delta\ut + \delta\vt) &= 0,\\
    \partial_1 \cA \delta t + \partial_2 \cA (\delta\ut + \delta\vt) &\ge 0,\\
    \partial_1 \cz \delta t + \partial_2 \cz (\delta\ut + \delta\vt) &=
    \delta \ut - \delta \vt, \\
    (\delta \ut, \delta \vt) &\in \Tcompl(\ut, \vt)
  \end{defarray}
\fi
\end{equation*}
with \emph{complementarity cone}
\begin{equation*}
  \Tcompl(\ut, \vt)
  \define
\ifcase0
  \begin{defarray}{(\delta \ut, \delta \vt)}
    \delta \ut_i = 0,\ i \in \setVt_+,\
    \delta \vt_i = 0,\ i \in \setUt_+, \\
    0 \le \delta \ut_i \compl \delta \vt_i \ge 0,\ i \in \setDt
  \end{defarray}
\else
  \begin{defarray}[r@{\medspace}l]{(\delta \ut, \delta \vt)}
    \delta \ut_i &= 0,\ i \in \setVt_+, \\
    \delta \vt_i &= 0,\ i \in \setUt_+, \\
    0 \le \delta \ut_i \compl \delta \vt_i &\ge 0,\ i \in \setDt
  \end{defarray}
\fi
.
\end{equation*}
Here, all partial derivatives are evaluated at $(t, \ut + \vt)$.
\end{definition}

Note that the MPCC-linearized cone was originally stated in \cite{Pang_Fukushima_1999} and \cite{Scheel_Scholtes_2000}, but was not further investigated there.
Moreover, we modified the definition in \cite{FlegelDiss} by introducing the complementarity cone which is studied in the next lemma.

\begin{lemma}
  The complementarity cone $\Tcompl(\hut, \hvt)$ is the \tcone
  and also the linearized cone to the complementarity set
  $\defset{(\ut, \vt)}{0 \le \ut \compl \vt \ge 0}$ at $(\hut, \hvt)$.
\end{lemma}
\begin{proof}
  Given a tangent vector
  $(\delta \ut, \delta \vt) = \lim \Inv\tau_k (\ut_k - \hut, \vt_k - \hvt)$
  where $0 \le \ut_k \compl \vt_k \ge 0$ and $\tau_k \searrow 0$,
  we have for $k$ large enough:
  \begin{align*}
    \ut_{ki} > 0,\ \vt_{ki} &= 0, & i \in \setUt_+\ (\hut_i &> 0,\ \hvt_i = 0), \\
    \ut_{ki} = 0,\ \vt_{ki} &> 0, & i \in \setVt_+\ (\hut_i &= 0,\ \hvt_i > 0), \\
    0 \le \ut_{ki} \compl \vt_{ki} &\ge 0, & i \in \setDt\ (\hut_i &= 0,\ \hvt_i = 0).
  \end{align*}
  This implies $(\delta \ut, \delta \vt) \in \Tcompl(\hut, \hvt)$.
  Conversely, every $(\delta \ut, \delta \vt) \in \Tcompl(\hut, \hvt)$
  is a tangent vector generated by the sequence
  $(\ut_k, v_k) = (\hut, \hvt) + \tau_k (\delta \ut, \delta \vt)$
  with $\tau_k = 1/k$, $k \in \mathbb N_{>0}$.
  The linearized cone clearly coincides with the \tcone.
\end{proof}

\begin{lemma} \label{le:hom-T-i}
  Given \eqref{eq:i-anf} with counterpart MPCC \eqref{eq:i-mpec},
  consider $(t, \zt) \in \Fabs$ with $\sigt=\sigt(t)$ and $(t, \ut, \vt) = \phi^{-1}(t,\zt)\in\Fmpec$ with associated index sets $\setUt_+$, $\setVt_+$ and $\setDt$.
  Define $\psi\: \Tmpec(t, \ut, \vt) \to \Tabs(t, \zt)$ and $\psi\: \Tlinmpec(t, \ut, \vt) \to \Tlinabs(t, \zt)$ as
  \begin{equation*}
    \psi(\delta t, \delta \ut, \delta \vt) = (\delta t, \delta \ut - \delta \vt) \qtextq{and} \Inv\psi(\delta t, \delta \zt) = (\delta t, \delpos\zt, \delneg\zt).
  \end{equation*}
  Here, $\delpos\zt, \delneg\zt$ map $\delta\zt$ into the complementarity cone via
  \bgroup
  \def\arraystretch{1.15}
  \begin{equation*}
    \delpos{\zt_i} =
    \begin{ccases}
      +\delta \zt_i, & i \in \setUt_+\ (\sigt_i > 0) \\ 0, & i \in \setVt_+\ (\sigt_i < 0) \\ {}[\delta \zt_i]^+, & i \in \setDt\ (\sigt_i = 0)
    \end{ccases}
    ,\quad
    \delneg{\zt_i} =
    \begin{ccases}
      0, & i \in \setUt_+\ (\sigt_i > 0) \\ -\delta \zt_i, & i \in \setVt_+\ (\sigt_i < 0) \\ {}[\delta \zt_i]^-, & i \in \setDt\ (\sigt_i = 0)
    \end{ccases}.
  \end{equation*}
  \egroup
  Then, both functions $\psi$ are homeomorphisms.
\end{lemma}
\begin{proof}
  First, consider $\psi\: \Tmpec(t, \ut, \vt) \to \Tabs(t, \zt)$:
  Given a \ifnum\FmtChoice=1 tangent \fi vector
  \begin{math}
    (\delta t, \delta \ut, \delta \vt) = \lim \Inv\tau_k (t_k - t, \ut_k - \ut, \vt_k - \vt) \in \Tmpec(t, \ut, \vt),
  \end{math}
  set
  \begin{math}
    (t_k, \zt_k) = \phi(t_k, \ut_k, \vt_k) = (t_k, \ut_k - \vt_k) \in \Fabs
  \end{math}
  to obtain
  \begin{equation*}
    \lim \frac{\zt_k - \zt}{\tau_k} = \lim \frac{(\ut_k - \ut) - (\vt_k - \vt)}{\tau_k} = \delta \ut - \delta \vt
    \implies (\delta t, \delta \ut - \delta \vt) \in \Tabs(t, \zt).
  \end{equation*}
  Conversely, given a vector $(\delta t, \delta \zt) = \lim \Inv\tau_k (t_k - t, \zt_k - \zt) \in \Tabs(t, \zt)$,
  define $(t_k, \ut_k, \vt_k) = \Inv\phi(t_k, \zt_k) = (t_k, [\zt_k]^+, [\zt_k]^-) \in \Fmpec$.
  Then, $\tau_k^{-1}((u_k-u)-(v_k-v))\to \delpos\zt - \delneg\zt$ holds.
  Thus, it remains to show $\tau_k^{-1}(u_k-u,v_k-v)\to (\delpos\zt,\delneg\zt)$ which is done componentwise:
  \begin{itemize}
  \item $i\in \setUt_+$: $\vt_i=0$ holds by feasibility and $\delneg\zt=0$ by definition.
  Thus, $(\ut_k)_i>0$ holds for $k$ large enough and by complementarity $(\vt_k)_i=0$ holds.
  Then, $\tau_k^{-1}((\ut_k)_i-\ut_i)\to \delpos\zt_i$ follows.
  \item $i\in\setVt_+$: $\tau_k^{-1}((\vt_k)_i-\vt_i)\to \delneg\zt_i$ follows as in the previous case.
  \item $i\in\setDt$ and $\delpos\zt_i>0$: $\delneg\zt_i=0$ holds by complementarity and so $\tau_k^{-1}((\ut_i)_k-(\vt_i)_k)\to \delpos\zt_i$.
  Then, $\tau_k^{-1}(\ut_i)_k\to \delpos\zt_i$ and $\tau_k^{-1}(\vt_i)_k\to 0$ because of sign constraints.
  \item $i\in\setDt$ and $\delneg\zt_i>0$:
  $\tau_k^{-1}(\ut_i)_k\to 0$ and $\tau_k^{-1}(\vt_i)_k\to \delneg\zt_i$ follow as in the previous case.
  \item $i\in\setDt$ and $\delpos\zt_i=\delneg\zt_i=0$: Then, $\tau_k^{-1}((\ut_i)_k-(\vt_i)_k)\to0$ holds. Because of sign constraints and complementarity, this can only hold if
  $\tau_k^{-1}(\ut_i)_k\to 0,\ \tau_k^{-1}(\vt_i)_k\to0$.
\end{itemize}
Altogether, this implies
  \begin{equation*}
    \lim \frac{(t_k - t, \ut_k - \ut, \vt_k - \vt)}{\tau_k} =
    (\delta t, \delpos\zt, \delneg\zt) \in \Tmpec(t, \ut, \vt).
  \end{equation*}
  By construction, $\psi$ and $\psi^{-1}$ are both continuous and inverse to each other.\\
  Second, consider $\psi\: \Tlinmpec(t, \ut, \vt) \to \Tlinabs(t, \zt)$:
  Given $(\delta t, \delta \ut, \delta \vt) \in \Tlinmpec(t, \ut, \vt)$,
  the vectors $\delta \zt = \delta \ut - \delta \vt$ and $\delta \zeta = \delta \ut + \delta \vt$ satisfy
\begin{equation*}
  \delta\zt_i =
  \begin{ccases}
    \delta\ut_i, & i \in \setUt_+ \\ -\delta\vt_i, & i \in \setVt_+ \\ \delta\ut_i - \delta\vt_i, & i \in \setDt
  \end{ccases}, \quad
  \delta\zeta_i =
  \begin{ccases}
    \delta\ut_i=\sigma_i\delta \zt_i, & i \in \setUt_+ \\ \delta\vt_i=\sigt_i\delta \zt_i, & i \in \setVt_+ \\ \delta\ut_i + \delta\vt_i=\abs{\delta \zt_i}, & i \in \setDt
  \end{ccases}.
\end{equation*}
  Thus, $(\delta t, \delta \zt) = \psi(\delta t, \delta \ut, \delta \vt) \in \Tlinabs(t, \zt)$.

 Conversely, the same case distinction yields
  \begin{math}
    (\delta t, \delta \ut, \delta \vt) =
    \Inv\psi(\delta t, \delta \zt)
    \in \Tlinmpec(\^t, \hut, \hvt)
  \end{math}
  for every $(\delta t, \delta \zt) \in \Tlinabs(\^t, \hzt)$.
  Again, $\psi$ and $\psi^{-1}$ are both continuous and inverse to each other by construction.
\end{proof}

\begin{definition}[Branch NLPs for \eqref{eq:i-mpec}, see \cite{Pang_Fukushima_1999}]\label{def:branch-mpec}
 Consider a feasible point $(\^t,\hut,\hvt)$ of \eqref{eq:i-mpec}
 with associated index sets $\setUt_+$, $\setVt_+$, and $\setDt$
 and choose $\setPt \subseteq \setDt$
 with complement $\bar\setPt = \setDt \setminus \setPt$.
 The branch problem NLP($\setPt$) is defined as
  \begin{align*}
    \sminst[t,\ut,\vt]{f(t)}
    & \ce(t,\ut+\vt) = 0,\\
    & \ci(t,\ut+\vt) \ge 0,\\
    & \cz(t,\ut+\vt) - (\ut - \vt) = 0,\tag{NLP($\setPt$)}\label{eq:branch-mpec}\\
    & 0 = \ut_i, \ 0 \le \vt_i, \ i\in\setVt_+\cup\setPt,\\
    & 0 \le \ut_i, \ 0 = \vt_i, \ i\in\setUt_+\cup\bar{\setPt}.
  \end{align*}
  The feasible set of \eqref{eq:branch-mpec},
  which always contains $(\^t,\hut,\hvt)$, is denoted by $\Fp$.
\end{definition}

Clearly, the homeomorphism $\phi$ can be restricted to the branch problems \eqref{eq:branch-anf} and \eqref{eq:branch-mpec} where $\setPt=\{i\in\alpt(\^t)\colon \sigt_i=-1 \}$.
Thus, the mapping $\phi_{\setPt}\: \Fp \to \Fsig$ defined by
\begin{equation*}
 \phi_{\setPt}\define \phi\vert_{\setPt} \qtextq{and} \phi^{-1}_{\setPt}\define \phi^{-1}\vert_{\Sigt}
\end{equation*}
is a homeomorphism.
\ifCmp\else
\begin{definition}[\TCone and Linearized Cone for \eqref{eq:branch-mpec}]\label{def:cones-branch-mpec}
 Given \eqref{eq:branch-mpec}, consider a feasible point $(t,\ut,\vt)$.
\fi
 The \tcone to $\Fp$ at $(t,\ut,\vt)$ is
  \begin{equation*}
    \Tp(t,\ut,\vt) \define
    \begin{defarray}{(\delta t,\delta\ut,\delta\vt)}
      \exists \tau_k \searrow 0,\ \Fp \ni (t_k, \ut_k,\vt_k) \to (t,\ut,\vt){:} \\
      \tau_k^{-1} (t_k - t, \ut_k - \ut, \vt_k - \vt) \to (\delta t, \delta\ut, \delta\vt)
    \end{defarray}.
  \end{equation*}
  The linearized cone is
  \begin{equation*}
    \Tlinp(t, \ut,\vt) \define
    \begin{defarray}{\begin{pmatrix} \delta t \\ \delta\ut \\ \delta\vt \end{pmatrix}}
      \partial_1 \ce \delta t +
      \partial_2 \ce (\delta\ut + \delta\vt) = 0, \\
      \partial_1 \cA \delta t +
      \partial_2 \cA (\delta\ut + \delta\vt) \ge 0, \\
      \partial_1 \cz \delta t +
      \partial_2 \cz (\delta\ut + \delta\vt) = \delta\ut - \delta \vt, \\
      0 = \delta\ut_i \ \text{for} \ i\in\setVt_+\cup\setP,\ 0 = \delta\vt_i \ \text{for} \ i\in\setUt_+\cup\bar\setP,\\
      0 \le \delta \ut_i \ \text{for} \ i\in \bar{\setP},\  0 \le \delta \vt_i \ \text{for} \ i\in\setP
    \end{defarray}.
  \end{equation*}
  Here, all partial derivatives are evaluated at $(t, \ut + \vt)$.
\ifCmp\else
\end{definition}
\fi

\begin{lemma} \label{le:hom-T-branch-i}
  Given \eqref{eq:branch-anf} and \eqref{eq:branch-mpec} with $\setPt=\defset{ i\in\alpt(\^t)}{\sigt_i=-1}$.
  Consider $(t, \zt) \in \Fsig$ and $(t,\ut,\vt)= \phi^{-1}_{\setPt}(t, \zt)$.
  Define $\psi_{\setPt}\define \psi\vert_{\Tp}$, $\psi^{-1}_{\setPt}\define \psi^{-1}\vert_{\Tsig}$ and $\psi_{\setPt}\define \psi\vert_{\Tlinp}$, $\psi^{-1}_{\setPt}\define \psi^{-1}\vert_{\Tlinsig}$.
  Then,
  \begin{equation*}
  \psi_{\setP}\: \Tp(t, \ut, \vt) \to \Tsig(t, \hzt) \qtextq{and} \psi_{\setP}\: \Tlinp(t, \hut, \hvt) \to \Tlinsig(t, \hzt)
  \end{equation*}
  are homeomorphisms.
\end{lemma}
\begin{proof}
  By construction and since $\alpt(\^t)=\setDt$, the following equalities of sets hold:
  \begin{align*}
   \setPt&=\{i\in\alpt(\^t): \sigt_i=-1\}, && \setVt_+=\{i\notin\alpt(\^t): \sigt_i=-1\},\\
   \bar{\setP}^t&=\{i\in\alpt(\^t): \sigt_i=+1\}, && \setUt_+=\{i\notin\alpt(\^t): \sigt_i=+1\}.
  \end{align*}
 Thus, the claim follows directly from \cref{le:hom-T-i}.
\end{proof}

\ifCmp\else
\begin{lemma}\label{le:decomp-cones-mpec}
\fi
Consider a feasible point $(t,\ut,\vt)$ of \eqref{eq:i-mpec} with associated branch problems \eqref{eq:branch-mpec}.
Then, the following decompositions of the \tcone
and of the abs-normal-linearized cone of \eqref{eq:i-mpec} hold
\ifCmp (for a proof see \cite{FlegelDiss})\fi:
 \begin{equation}
   \label{eq:decomp-cones-mpec}
  \Tmpec(t,\ut,\vt)=\bigcup_{\setPt} \Tp(t,\ut,\vt)
  \qtextq{and}
  \Tlinmpec(t,\ut,\vt)=\bigcup_{\setPt} \Tlinp(t,\ut,\vt).
 \end{equation}
\ifCmp
The following inclusions are also proved in \cite{FlegelDiss}:
\fi
 \begin{equation*}
  \Tmpec(t,\ut,\vt)\subseteq \Tlinmpec(t,\ut,\vt) \qtextq{and} \Tmpec(t,\ut,\vt)^* \supseteq \Tlinmpec(t,\ut,\vt)^*.
 \end{equation*}
\ifCmp\else
\end{lemma}
\begin{proof}
A proof may be found in \cite{FlegelDiss}.
\end{proof}
\fi

In general, the converses do not hold. This motivates the definition of MPCC-ACQ and MPCC-GCQ.

\ifCmp
\begin{definition}
  [Abadie's and Guignard's Constraint Qualifications for \eqref{eq:i-mpec}, see \cite{FlegelDiss}]
  \label{def:mpec-acq}
  \label{def:mpec-gcq}
  Consider a feasible point $(t,\ut,\vt)$ of \eqref{eq:i-mpec}.
  We say that \emph{Abadie's Constraint Qualification for MPCC (MPCC-ACQ)}
  holds at $(t,\ut,\vt)$ if $\Tmpec(t,\ut,\vt) = \Tlinmpec(t,\ut,\vt)$,
  and that \emph{Guignard's Constraint Qualification for MPCC (MPCC-GCQ)}
  holds at $(t,\ut,\vt)$ if $\Tmpec(t,\ut,\vt)^* = \Tlinmpec(t,\ut,\vt)^*$.
\end{definition}
The decomposition \eqref{eq:decomp-cones-mpec} and its dualization imply
that both MPCC-CQs hold if the corresponding CQ holds for all branch problems.
\begin{theorem}[ACQ/GCQ for all \eqref{eq:branch-mpec}
  implies MPCC-ACQ/MPCC-GCQ for \eqref{eq:i-mpec}]
  \label{th:branch-acq_mpec-acq}
  \label{th:branch-gcq_mpec-gcq}
  Consider a feasible point $(t,\ut,\vt)$ of \eqref{eq:i-mpec}.
  Then, MPCC-ACQ respectively MPCC-GCQ holds at $(t,\ut,\vt)$ if
  ACQ respectively GCQ holds for all \eqref{eq:branch-mpec} at $(t,\ut,\vt)$.
\end{theorem}
\fi

\subsection{Counterpart MPCC for the Abs-Normal NLP with Inequality Slacks}\label{subsec:e-mpec}

\ifCmp
By \cref{def:i-mpec},
the \emph{counterpart MPCC} of the non-smooth NLP \eqref{eq:e-anf} reads:
\else
Using the same approach as in the preceding paragraph, we restate the counterpart MPCC of \eqref{eq:e-anf}.

\begin{definition}[Counterpart MPCC of \eqref{eq:e-anf}]
  The \emph{counterpart MPCC} of the non-smooth NLP \eqref{eq:e-anf} reads:
\fi
  \begin{align*}
     \sminst[t,w,\ut,\vt,\uw,\vw]{f(t)}
     & \ce(t,\ut+\vt)=0,\\
     & \ci(t,\ut+\vt) - (\uw+\vw) = 0, \\
     & \cz(t,\ut+\vt) - (\ut-\vt) = 0,\tag{E-MPCC} \label{eq:e-mpec}\\
     & w - (\uw-\vw) = 0, \\
     & 0 \le \ut \compl \vt \ge 0, \quad 0 \le \uw \compl \vw \ge 0,
   \end{align*}
   where $\ut, \vt \in \R^{s_t}$ and $\uw, \vw \in \R^{m_2}$.
   The feasible set is denoted by $\Fempec$ and is a lifting of $\Fmpec$.
\ifCmp\else
\end{definition}
\fi

Clearly, the homeomorphism between $\Fmpec$ and $\Fabs$
extends to $\Fempec$ and $\Feabs$. It is given by
  \begin{align*}
    \_\phi(t, w, \ut, \vt, \uw, \vw) &= (t, w, \ut - \vt, \uw - \vw),\\
    \Inv{\_\phi}(t, w, \zt, \zw) &= (t, w, [\zt]^+, [\zt]^-, [\zw]^+, [\zw]^-).
  \end{align*}

Just like in the abs-normal case, problem \eqref{eq:e-mpec} is a special case of \eqref{eq:i-mpec}.
\ifCmp
Hence, we obtain the following material by specializing the definitions
and results for \eqref{eq:i-mpec}.

By \cref{def:cones-i-mpec}, the \tcone to $\Fempec$ at $y$ reads
\else
Hence, we obtain the next lemmas from the corresponding definitions and lemmas for \eqref{eq:i-mpec}.

\begin{lemma}[\TCone and MPCC-Linearized Cone for \eqref{eq:e-mpec}]\label{le:cones-e-mpec}
Consider a feasible point $y = (t,w, \ut,\vt,\uw,\vw)$ of \eqref{eq:e-mpec}.
The \tcone to $\Fempec$ at $y$ reads
\fi
\begin{equation*}
  \Tempec(y) =
  \begin{defarray}{\delta}
    \exists \tau_k \searrow 0,\
    \Fempec \ni y_k = (t_k, w_k, \ut_k, \vt_k, \uw_k, \vw_k) \to y{:} \\
    \tau_k^{-1} (y_k - y) \to
    \delta = (\delta t, \delta w, \delta\ut, \delta\vt, \delta\uw, \delta\vw)
  \end{defarray}
  .
\end{equation*}
The MPCC-linearized cone reads
\begin{equation*}
  \Tlinempec(\^y) =
  \ifcase0
  \begin{defarray} \delta
    \partial_1\ci \delta t + \partial_2\ci (\delta\ut + \delta\vt)
    = \delta\uw + \delta\vw, \
    \delta w = \delta\uw - \delta\vw, \\
    (\delta t, \delta \ut \delta \vt) \in \Tlinmpec(t,\ut,\vt), \
    (\delta \uw, \delta \vw) \in \Tcompl(\huw, \hvw)
  \end{defarray}
  \or
  \begin{defarray}[r@{\medspace}l]{\delta}
    (\delta t, \delta \ut \delta \vt) &\in \Tlinmpec(t,\ut,\vt), \\
    \partial_1\ci \delta t + \partial_2\ci (\delta\ut + \delta\vt)
    &= \delta\uw + \delta\vw, \\
    \delta w &= \delta\uw - \delta\vw, \\
    (\delta \uw, \delta \vw) &\in \Tcompl(\huw, \hvw)
  \end{defarray}
  \fi
  .
\end{equation*}
Here, all partial derivatives are evaluated at $(t,\ut+\vt)$.
\ifCmp
The associated homeomorphisms of \cref{le:hom-T-i},
\fi
  \begin{gather*}
    \_\psi\: \Tempec(t, w, \ut, \vt, \uw, \vw) \to \Teabs(t, w, \zt, \zw), \\
    \_\psi\: \Tlinempec(t, w, \ut, \vt, \uw, \vw) \to \Tlineabs(t, w, \zt, \zw),
  \end{gather*}
  \ifCmp now take the form\fi
  \begin{align*}
    \_\psi(\delta t, \delta w, \delta \ut, \delta \vt, \delta \uw, \delta \vw) &= (\delta t, \delta w, \delta \ut - \delta \vt, \delta \uw - \delta \vw),\\
    \Inv{\_\psi}(\delta t, \delta w, \delta \zt, \delta \zw) &= (\delta t, \delta w, \delpos\zt, \delneg\zt, \delpos\zw, \delneg\zw).
  \end{align*}
 Given $\^y = (\^t,\^w,\hut,\hvt,\huw,\hvw)$, a feasible point of \eqref{eq:e-mpec} with associated index sets $\setUt_+$, $\setVt_+$, $\setDt$, $\setUw_+$, $\setVw_+$, and $\setDw$,
 choose $\setPt\subseteq\setDt$ as well as $\setPw\subseteq\setDw$ and set $\setPtw=\setPt\cup\setPw$.
 The branch problem NLP($\setPtw$)
 \ifCmp of \cref{def:branch-mpec} then \fi reads
  \begin{align*}
    \sminst[t,w,\ut,\vt,\uw, \vw]{f(t)}
    & \ce(t,\ut+\vt) = 0, \quad
    \ci(t,\ut+\vt) - (\uw+\vw) = 0, \\
    & \cz(t,\ut+\vt) - (\ut-\vt) = 0, \quad
    w - (\uw-\vw) = 0, \\
    & 0 = \ut_i,\ 0 \le \vt_i, \ i\in \setVt_+\cup\setPt, \\
    & 0 \le \ut_i,\ 0 = \vt_i, \ i\in \setUt_+\cup\bar\setPt,
    \tag{NLP($\setPtw$)}\label{eq:branch-e-mpec}\\
    & 0 = \uw_i,\ 0 \le \vw_i, \ i\in \setVw_+\cup\setPw,\\
    & 0 \le \uw_i,\ 0 = \vw_i, \ i\in \setUw_+\cup\bar\setPw.
  \end{align*}
  The feasible set of \eqref{eq:branch-e-mpec},
  which always contains $\^y$, is denoted by $\Fep$ and is a lifting of $\Fp$.
\ifCmp\else
\end{lemma}
\begin{proof}
This follows from \cref{def:branch-mpec}.
\end{proof}
\fi

Again, the homeomorphism between feasible sets can be restricted to the \ifnum\FmtChoice=2 respective \fi branch problems \eqref{eq:branch-e-anf} and \eqref{eq:branch-e-mpec} where $\setPt=\{i\in\alpt(\^t)\colon \sigt_i=-1 \}$ and $\setPw=\{i\in\alpw(\^w)\colon \sigw_i=-1 \}$.
Thus, the mapping $\_\phi_{\setPtw}\: \Fep \to \Fesig$ given as
\begin{equation*}
\_\phi_{\setPtw}\define \_\phi\vert_{\setPtw} \qtextq{and} \_\phi^{-1}_{\setPtw}\define \_\phi^{-1}\vert_{\Sigtw}
\end{equation*}
is a homeomorphism.

\ifCmp\else
\begin{lemma}[\TCone and Linearized Cone for \eqref{eq:branch-e-mpec}]\label{def:cones-branch-e-mpec}
 Consider a feasible point $y=(t,w,\ut,\vt,\uw,\vw)$ of \eqref{eq:branch-e-mpec}.
\fi
 The \tcone to $\Fep$ at $y$ reads
  \begin{equation*}
    \Tep(y) =
    \begin{defarray}{\delta}
      \exists \tau_k \searrow 0,\ \Fep \ni (t_k,w_k,\ut_k,\vt_k,\uw_k,\vw_k)
      \to (t,w,\ut,\vt,\uw,\vw){:} \\
      \tau_k^{-1} (t_k - t, w_k-w, \ut_k - \ut, \vt_k - \vt, \uw_k - \uw, \vw_k - \vw) \to \delta
    \end{defarray}
  \end{equation*}
  where
  $\delta = (\delta t, \delta w, \delta\ut, \delta\vt, \delta\uw, \delta\vw)$.
  The linearized cone reads
  \begin{equation*}
    \Tlinep(y) =
    \begin{defarray}{\delta}
      (\delta t,\delta \ut, \delta \vt)\in\Tlinp, \\
      \partial_1 \ci \delta t +
      \partial_2 \ci (\delta\ut + \delta\vt) = \delta\uw + \delta\vw, \
      \delta w = \delta \uw - \delta \vw, \\
      0 = \delta\uw_i \ \text{for} \ i\in\setVw_+\cup\setPw,\ 0 = \delta\vw_i \ \text{for} \ i\in\setUw_+\cup\bar\setPw,\\
      0 \le \delta \uw_i \ \text{for} \ i\in \bar{\setPw},\  0 \le \delta \vw_i \ \text{for} \ i\in\setPw
    \end{defarray}
    .
  \end{equation*}
  Here, all partial derivatives are evaluated at $(t, \ut + \vt)$.
\ifCmp
The associated cone homeomorphisms of \cref{le:hom-T-branch-i}
are now obtained as follows.
\fi
  Given \eqref{eq:branch-e-anf} and \eqref{eq:branch-e-mpec} with $\setPt=\defset{ i\in\alpt(\^t)}{\sigt_i=-1}$ and $\setPw=\defset{ i\in\alpw(\^w)}{\sigw_i=-1}$,
  consider $(t, w, \zt, \zw) \in \Feabs$ and $(t, w, \ut, \vt, \uw, \vw)=\_\phi^{-1}(t, w, \zt, \zw)$.
  Define $\bar\psi_{\setPtw}\define \bar\psi\vert_{\Tep}$, $\bar\psi^{-1}_{\setPtw}\define \bar\psi\vert_{\Tesig}$
  and $\bar\psi_{\setPtw}\define \bar\psi\vert_{\Tlinep}$, $\bar\psi^{-1}_{\setPtw}\define \bar\psi\vert_{\Tlinesig}$.
  Then, we have 
  \begin{align*}
  \bar\psi_{\setPtw}\: \Tep(t, w, \ut, \vt, \uw, \vw) \to \Tesig(t, w, \zt, \zw),\\
  \bar\psi_{\setPtw}\: \Tlinep(t, w, \ut, \vt, \uw, \vw) \to \Tlinesig(t, w, \zt, \zw).
  \end{align*}
\ifCmp\else
\end{lemma}
\begin{proof}
 This follows directly from \cref{le:hom-T-branch-i}.
\end{proof}
\fi

By applying \ifCmp
\eqref{eq:decomp-cones-mpec} \fi
to \eqref{eq:e-mpec} with associated branch problems \eqref{eq:branch-e-mpec},
we obtain the following decomposition of cones at $y = (t,w,\ut,\vt,\uw,\vw)$:
 \begin{equation*}
  \Tempec(y)=\bigcup_{\setPtw} \Tep(y) \qtextq{and}
  \Tlinempec(y)=\bigcup_{\setPtw} \Tlinep(y).
 \end{equation*}
Moreover, the \tcone is contained in the linearized cone
and the converse holds for the dual cones:
 \begin{equation*}
   \Tempec(y) \subseteq \Tlinempec(y) \qtextq{and}
   \Tempec(y)^* \supseteq \Tlinempec(y)^*.
 \end{equation*}
Once again, the converses do not hold in general and we consider
Abadie's and Guignard's Constraint Qualifications for \eqref{eq:e-mpec}
at $y = (t,w, \ut,\vt, \uw,\vw)$.
\ifCmp
Recalling \cref{def:mpec-acq}, MPCC-ACQ and MPCC-GCQ simply read
\begin{equation*}
  \Tempec(y) = \Tlinempec(y)
  \qtextq{and}
  \Tempec(y)^* = \Tlinempec(y)^*
  .
\end{equation*}
\fi

\begin{remark}
Let
\begin{equation*}
  W(t,\ut,\vt)=\defset{(w,\uw,\vw)}{\abs{w}=\ci(t,\ut+\vt),\uw=[w]^+,\vw=[w]^-}.
\end{equation*}
 Due to symmetry, the above equality of cones (respectively dual cones)
 \ifnum\FmtChoice=2 clearly \fi holds for all elements $(w,\uw,\vw)\in W(t,\ut,\vt)$
 if it holds for any element.
\end{remark}

\ifCmp
Now \cref{th:branch-acq_mpec-acq} reads as follows.
\begin{theorem}[ACQ/GCQ for all \eqref{eq:branch-e-mpec}
  implies MPCC-ACQ/MPCC-GCQ for \eqref{eq:e-mpec}]
  \label{th:e-branch-acq_mpec-acq}
  \label{th:e-branch-gcq_mpec-gcq}
  Consider a feasible point $y = (t,w, \ut,\vt, \uw,\vw)$ of \eqref{eq:e-mpec}
  with \ifnum\FmtChoice<2 associated \fi
   branch problems \eqref{eq:branch-e-mpec}.
  Then, MPCC-ACQ respectively MPCC-GCQ holds for \eqref{eq:e-mpec} at $y$
  if ACQ respectively GCQ holds for all \eqref{eq:branch-e-mpec} at $y$.
\end{theorem}
\fi

\subsection{Relations of MPCC-CQs for Different Formulations}

In this paragraph we prove relations between constraint qualifications for the two different formulations \eqref{eq:i-mpec} and \eqref{eq:e-mpec}.
Some relations follow from the results in the previous section and in the two following sections. 

\begin{theorem}\label{th:acq}
 MPCC-ACQ for \eqref{eq:i-mpec} holds at $(t,\ut,\vt)\in\Fmpec$ if and only if
 MPCC-ACQ for \eqref{eq:e-mpec} holds at $(t,w,\ut,\uw,\vt,\vw)\in\Fempec$
 for any (and hence all) $(w, \uw, \vw) \in W(t, \ut, \vt)$.
\end{theorem}
\begin{proof}
 This follows immediately from \cref{th:akq}, \cref{th:akq-acq-i} and \cref{th:acq-akq-e}.
\end{proof}

\begin{theorem}\label{th:gcq}
 MPCC-GCQ for \eqref{eq:i-mpec} holds at $(t,\ut,\vt)\in\Fmpec$ if
 MPCC-GCQ for \eqref{eq:e-mpec} holds at $(t,w,\ut,\vt,\uw,\vw)\in\Fempec$
 for any (and hence all) $(w, \uw, \vw) \in W(t, \ut, \vt)$.
\end{theorem}
\begin{proof}
 The inclusion $\Tmpec(t, \ut, \vt)^* \supseteq \Tlinmpec(t, \ut, \vt)^*$ holds always.
 Thus, it is left to show that
\begin{equation*}
  \Tmpec(t,\ut, \vt)^* \subseteq \Tlinmpec(t, \ut, \vt)^*.
\end{equation*}
 Let $\omega=(\omega t,\omega \ut, \omega \vt)\in \Tmpec(t, , \ut, \vt)^*$,
 i.e. $\omega^T\delta\ge 0$ for all $\delta=(\delta t,\delta \ut,\delta \vt)\in \Tmpec(t, \ut, \vt)$.
 Then, \ifnum\FmtChoice=0 define \else let \fi $\tilde{\omega}=(\omega t,0,\omega \ut, \omega \vt,0,0)$ to obtain $\tilde{\omega}^T\tilde \delta=\omega^T\delta\ge 0$ for all $\tilde \delta\in \Tempec(t,w, \ut, \vt,\uw,\vw)$ where $w\in W(t)$ is arbitrary.
 By assumption, we have $\tilde{\omega}^T\tilde \delta\ge 0$ for all $\tilde \delta\in \Tlinempec(t,w, \ut, \vt,\uw,\vw)$
 which implies $\omega^T\delta=\tilde{\omega}^T\tilde \delta\ge 0$ for all $\delta\in \Tlinmpec(t, \ut, \vt)$.
\end{proof}

The converse of the previous theorem is unlikely to hold, but we do not know how to construct a counterexample.
However, equivalence of ACQ or GCQ for corresponding branch problems holds.

\begin{theorem}\label{th:mpec-branch-acq}
 ACQ for \eqref{eq:branch-mpec} holds at $(t,\ut,\vt)\in\Fp$ if and only if
 ACQ for \eqref{eq:branch-e-mpec} holds at $(t,w,\ut,\vt,\uw,\vw)\in\Fep$
 for any (and hence all) $(w, \uw, \vw) \in W(t, \ut, \vt)$.
\end{theorem}
\begin{proof}
 This follows immediately from \cref{th:branch-acq}, \cref{th:branch-acq-i} and \cref{th:branch-acq-e}.
\end{proof}

\begin{theorem}\label{th:mpec-branch-gcq}
 GCQ for \eqref{eq:branch-mpec} holds at $(t,\ut,\vt)\in\Fp$ if and only if
 GCQ for \eqref{eq:branch-e-mpec} holds at $(t,w,\ut,\vt,\uw,\vw)\in\Fep$
 for any (and hence all) $(w, \uw, \vw) \in W(t, \ut, \vt)$.
\end{theorem}
\begin{proof}
 This follows immediately from \cref{th:branch-gcq}, \cref{th:branch-gcq-i} and \cref{th:branch-gcq-e}.
\end{proof}

\section{Kink Qualifications and MPCC Constraint Qualifications}
\label{sec:qualifications}

In this section we show relations between abs-normal NLPs and counterpart MPCCs.
Here, we consider both treatments of inequality constraints.

\subsection{Relations of General Abs-Normal NLP and MPCC}
In the following the variables $x$ and $z$ instead of $t$ and $\zt$ are used.
Thus, the  abs-normal NLP \eqref{eq:i-anf} reads:
  \begin{equation*}
    \minst[x,z]{f(x)}
    \ce(x,\abs{z})=0,\quad
    \ci(x,\abs{z}) \ge 0,\quad
    \cz(x,\abs{z})-z=0.
  \end{equation*}
The counterpart MPCC \eqref{eq:i-mpec} becomes:
  \begin{align*}
    \sminst[x,u,v]{f(x)}
    & \ce(x,u+v)=0, \quad
    \ci(x,u+v) \ge 0,\\
    & \cz(x,u+v)-(u-v)=0, \quad
    0 \le u \compl v \ge 0.
  \end{align*}

Then, the subsequent relations of kink qualifications and MPCC constraint qualifications can be shown.

\begin{theorem}[AKQ for \eqref{eq:i-anf} $\iff$ MPCC-ACQ for \eqref{eq:i-mpec}]\label{th:akq-acq-i}
  AKQ for \eqref{eq:i-anf} holds at $(x,z(x))\in\Fabs$ if and only if MPCC-ACQ for \eqref{eq:i-mpec} holds at
  $(x,u,v)=(x,[z(x)]^+,[z(x)]^-)\in\Fmpec$.
\end{theorem}
\begin{proof}
\ifcase0
We need to show
\begin{equation*}
  \Tabs(x, z) = \Tlinabs(x, z) \iff \Tmpec(x, u, v) = \Tlinmpec(x, u, v).
\end{equation*}
This is obvious from the homeomorphisms $\psi$ in \cref{le:hom-T-i}.
\or
By \cref{le:i-cones} and \cref{le:i-cones-mpec} we always have
\begin{equation*}
 \Tabs(x,z)\subseteq\Tlinabs(x,z) \qtextq{and} \Tmpec(x,u,v)\subseteq\Tlinmpec(x,u,v).
\end{equation*}
Thus, we just need to prove
\begin{equation*}
 \Tabs(x,z)=\Tlinabs(x,z) & \implies \Tmpec(x,u,v)\supseteq\Tlinmpec(x,u,v),\\
 \Tmpec(x,u,v)=\Tlinmpec(x,u,v) & \implies \Tabs(x,z)\supseteq\Tlinabs(x,z).
\end{equation*}
We start with the implication ``$\Rightarrow$'' and consider $\delta=(\delta x, \delta u,\delta v)\in\Tlinmpec(x,u,v)$.
Then, we set
\begin{equation*}
  \delta z_i \define \delta u_i - \delta v_i =
  \begin{ccases}
    +\delta u_i, & i \in \setU_+ \\
    -\delta v_i, & i \in \setV_+ \\
    \delta u_i - \delta v_i, & i \in \setD
  \end{ccases}
  \qtextq{with}
  0 \le \delta u_i \compl \delta v_i \ge 0 \text{ for } i \in \setD,
\end{equation*}
and obtain $\tilde\delta=(\delta x,\delta z)\in\Tlinabs(x,z)$.
This is because
\begin{equation*}
  \delta u_i + \delta v_i =
  \begin{ccases}
    \delta u_i, & i \in \setU_+ \\
    \delta v_i, & i \in \setV_+ \\
    \delta u_i + \delta v_i, & i \in \setD
  \end{ccases}
  =
  \begin{ccases}
    \sigma_i\delta z_i, & \sigma_i = +1 \\
    \sigma_i\delta z_i, & \sigma_i = -1 \\
    \abs{\delta z_i}, & \sigma_i = 0
  \end{ccases}
  = \delta\zeta \qtextq{for} \sigma = \sigma(x).
\end{equation*}
By assumption, $\tilde\delta\in\Tabs(x,z)$ and there exist sequences $\Fabs\ni(x_k,z_k)\to (x,z)$
and $\tau_k\searrow 0$ such that $\tau_k^{-1}(x_k-x,z_k-z)\to (\delta x,\delta z)$.
We set $u_k=[z_k]^+$ and $v_k=[z_k]^-$ and have $\tau_k^{-1}((u_k-u)-(v_k-v))\to \delta u-\delta v$.
Thus, it remains to show $\tau_k^{-1}(u_k-u,v_k-v)\to(\delta u,\delta v)$.
We show this componentwise:
\begin{itemize}
\item $i\in \setU_+$:
we have $v_i=0$ by feasibility and $\delta u_i\neq 0$ and $\delta v_i=0$ by definition of $\Tlinmpec$.
Thus, $(u_k)_i>0$ holds for $k$ large enough and by complementarity $(v_k)_i=0$ holds.
Then, $\tau_k^{-1}((u_k)_i-u_i)\to {\delta u}_i$ follows.
\item $i\in\setV_+$: $\tau_k^{-1}((v_k)_i-v_i)\to {\delta v}_i$ follows as in the previous case .
\item $i\in\setD$ and $\delta u_i>0$:
we have $\delta v_i=0$ by complementarity and so $\tau_k^{-1}((u_i)_k-(v_i)_k)\to \delta u_i$.
Then, $\tau_k^{-1}(u_i)_k\to \delta u_i$ and $\tau_k^{-1}(v_i)_k\to 0$ because of sign constraints.
\item $i\in\setD$ and $\delta v_i>0$:
$\tau_k^{-1}(u_i)_k\to 0$ and $\tau_k^{-1}(v_i)_k\to \delta v_i$ follow as in the previous case.
\item $i\in\setD$ and ${\delta u}_i={\delta v}_i=0$, we have
$\tau_k^{-1}((u_i)_k-(v_i)_k)\to0$. Because of sign constraints and complementarity, this can only hold if
$\tau_k^{-1}(u_i)_k\to 0,\ \tau_k^{-1}(v_i)_k\to0$.
\end{itemize}
Altogether we have $u_k-u\to\delta u$ and $v_k-v\to\delta v$, i.e. $\delta\in\Tmpec(x,u,v)$.\\
Now, we prove the implication ``$\Leftarrow$''. To this end, consider $\delta=(\delta x,\delta z)\in \Tlinabs(x,z)$.
Then, set
\begin{equation*}
  \delta u_i \define
  \begin{ccases}
    \delta z_i & \sigma_i = +1 \\
    0, & \sigma_i = -1 \\
    {}[\delta z_i]^+, & \sigma_i = 0
  \end{ccases}
  \qtextq{and}
  \delta v_i \define
  \begin{ccases}
    0 & \sigma_i = +1 \\
    -\delta z_i, & \sigma_i = -1 \\
    {}[\delta z_i]^-, & \sigma_i = 0
  \end{ccases}
\end{equation*}
to obtain $\delta z = \delta  u - \delta v$ and $\delta \zeta= \delta u + \delta v$.
Thus, $\tilde\delta=(\delta x,\delta u,\delta v)\in\Tlinmpec(x,u,v)$ and by assumption $\tilde d\in\Tmpec$.
Thus, there exist $\Fmpec\ni(x_k,u_k,v_k)\to(x,u,v)$ and $\tau_k\searrow 0$ with
$\tau_k^{-1}(x_k-x,u_k-u,v_k-v)\to(\delta x,\delta u, \delta v)$.
With $z_k=u_k+v_k$ we have $z_k\to z=u+v$ and $\tau_k^{-1}(z_k-z)\to \delta u-\delta v =\delta z$,
i.e. $d\in\Tabs(x,z)$.
\fi
\end{proof}

\begin{theorem}[MPCC-GCQ for \eqref{eq:i-mpec} implies GKQ for \eqref{eq:i-anf}]\label{th:gkq-gcq-i}
  GKQ for \eqref{eq:i-anf} holds at $(x,z(x))\in\Fabs$ if MPCC-GCQ for \eqref{eq:i-mpec} holds at
  $(x,u,v)=(x,[z(x)]^+,[z(x)]^-)\in\Fmpec$.
\end{theorem}
\begin{proof}
The inclusion $\Tlinabs(x,z)^*\subseteq\Tabs(x,z)^*$ hold always by \cref{le:i-cones}.
Thus, we just have to show
\begin{equation*}
  \Tabs(x, z)^* \subseteq \Tlinabs(x, z)^*.
\end{equation*}
Consider $\omega=(\omega x,\omega z)\in\Tabs(x,z)^*$, i.e. $\omega^T\delta \ge 0$ for all $\delta=(\delta x,\delta z)\in\Tabs(x,z)$.
Set $\~\omega = (\omega x, \omega z, -\omega z)$.
For every $\delta \in \Tabs(x, z)$ we then have
\begin{equation*}
  \~\omega^T \Inv\psi(\delta) =
  \omega x^T \delta x + \omega z^T \delpos z - \omega z^T \delneg z =
  \omega x^T \delta x + \omega z^T \delta z =
  \omega^T \delta \ge 0.
\end{equation*}
This means $\~\omega \in \Tmpec(x, u, v)^*$ and hence,
by assumption, $\~\omega \in \Tlinmpec(x, u, v)^*$.
We thus have $\omega^T \delta = \~\omega^T \Inv\psi(\delta) \ge 0$
for every $\delta \in \Tlinabs(x, z)$,
which means $\omega\in\Tlinabs(x,z)^*$.
\end{proof}

The converse is unlikely to hold, although we are not, at this time, aware of a counterexample.
Once again, moving to the branch problems allows to exploit additional sign information.

\begin{theorem}[ACQ for \eqref{eq:branch-anf} $\iff$ ACQ for \eqref{eq:branch-mpec}]\label{th:branch-acq-i}
  ACQ for \eqref{eq:branch-anf} holds at $(x,z(x))\in\Fsig$ if and only if ACQ for the corresponding \eqref{eq:branch-mpec} holds at
  $(x,u,v)=(x,[z(x)]^+,[z(x)]^-)\in\Fp$.
\end{theorem}
\begin{proof}
We need to show
\begin{equation*}
  \Tsig(x, z) = \Tlinsig(x, z) \iff \Tp(x, u, v) = \Tlinp(x, u, v).
\end{equation*}
This is obvious from the homeomorphisms $\psi_{\setP}$ in \cref{le:hom-T-branch-i}.
\end{proof}

\begin{theorem}[GCQ for \eqref{eq:branch-anf} $\iff$ GCQ for \eqref{eq:branch-mpec}]\label{th:branch-gcq-i}
  GCQ for \eqref{eq:branch-anf} holds at $(x,z(x))\in\Fsig$ if and only if GCQ for the corresponding \eqref{eq:branch-mpec} holds at
  $(x,u,v)=(x,[z(x)]^+,[z(x)]^-)\in\Fp$.
\end{theorem}
\begin{proof}
The inclusions $\Tlinp(x,u,v)^*\subseteq\Tp(x,u,v)^*$ and $\Tlinsig(x,z)^*\subseteq\Tsig(x,z)^*$ hold always.
Thus, we just have to show
\begin{equation*}
  \Tsig(x, z)^* \supseteq \Tlinsig(x, z)^* \iff \Tp(x,u,v)^* \supseteq \Tlinp(x,u,v)^*.
\end{equation*}
First, we show the implication ``$\Rightarrow$''.
Consider $\omega=(\omega x,\omega u,\omega v)\in\Tp(x,u,v)^*$, i.e. $\omega^T\delta \ge 0$ for all $\delta=(\delta x,\delta u,\delta v)\in\Tp(x,u,v)$.
Set $\~\omega = (\omega x, \omega z)$ with
\begin{equation*}
 \omega z_i=
 \begin{cases}
  +\omega u_i, & i\in\setU_+\cup \setP,\\
  -\omega v_i, & i\in\setV_+\cup \bar\setP.
 \end{cases}
\end{equation*}
This leads to
\begin{equation*}
  \~\omega^T \psi_{\setP}(\delta) =
  \omega x^T \delta x + \omega z^T (\delta u - \delta v) =
  \omega x^T \delta x + \omega u^T \delta u + \omega v^T \delta v =
  \omega^T \delta \ge 0
\end{equation*}
for every $\delta \in \Tp(x, u, v)$, i.e. $\~\omega \in \Tsig(x, z)^*$.
Then, the assumption yields $\~\omega \in \Tlinsig(x, z)^*$. As we have $\omega^T \delta = \~\omega^T \psi_{\setP}(\delta) \ge 0$
for every $\delta \in \Tlinp(x, u,v)$, we obtain $\omega\in\Tlinp(x,u,v)^*$.
The reverse implication follows as in \cref{th:gkq-gcq-i}.
\end{proof}

\subsection{Relations of Abs-Normal NLP and MPCC with Inequality Slacks}

Now, the relations for the slack reformulations are stated.
These are special cases of the general problem formulations,
hence we \ifCmp obtain the following four theorems that correspond
to \crefrange{th:akq-acq-i}{th:branch-gcq-i}.\fi

\begin{theorem}[AKQ for \eqref{eq:e-anf} $\iff$ MPCC-ACQ for \eqref{eq:e-mpec}] \label{th:acq-akq-e}
  AKQ for \eqref{eq:e-anf} holds at $(x,z(x)) \in \Feabs$
  if and only if MPCC-ACQ for \eqref{eq:e-mpec} holds at
  $(x, u, v)=(x,[z(x)]^+,[z(x)]^-) \in \Fempec$.
\end{theorem}

\begin{theorem}[MPCC-GCQ for \eqref{eq:e-mpec} implies GKQ for \eqref{eq:e-anf}]
  \label{th:gcq-gkq-e}
  GKQ for \eqref{eq:e-anf} holds at $(x,z(x)) \in \Feabs$
  if MPCC-GCQ for \eqref{eq:e-mpec} holds at
  $(x, u, v)=(x,[z(x)]^+, [z(x)]^-)\in \Fempec$.
\end{theorem}

The converse is unlikely to hold, but to date we are not aware of a counterexample.

\begin{theorem}[ACQ for \eqref{eq:branch-e-anf} $\iff$ ACQ for  \eqref{eq:branch-e-mpec}]\label{th:branch-acq-e}
  ACQ for \eqref{eq:branch-e-anf} at $(x,z(x))\in\Fesig$ is equivalent to ACQ for the corresponding \eqref{eq:branch-e-mpec} at
  $(x,u,v)=(x,[z(x)]^+,[z(x)]^-)\in\Fep$.
\end{theorem}

\begin{theorem}[GCQ for \eqref{eq:branch-e-anf} $\iff$ GCQ for  \eqref{eq:branch-e-mpec}]\label{th:branch-gcq-e}
  GCQ for \eqref{eq:branch-e-anf} at $(x,z(x))\Fesig$ is equivalent to GCQ for
  \ifnum\FmtChoice=1\else the corresponding \fi \eqref{eq:branch-e-mpec} at
  $(x,u,v)=(x,[z(x)]^+,[z(x)]^-)\in\Fep$.
\end{theorem}

All the relations discussed in \crefrange{sec:anf-formulations}{sec:qualifications}
are illustrated in \cref{fig:relations}.
In the inner square (containing \eqref{eq:i-anf} and \eqref{eq:e-anf} as well as the counterpart MPCCs \eqref{eq:i-mpec} and \eqref{eq:e-mpec}) there are four single-headed arrows,
which indicate that only one direction has been proved
and we do not know if the converses hold as well.
Therefore we considered the branch problems given on the outer right and left in the figure.
Since ACQ respectively GCQ for all branch problems imply the corresponding
kink qualification or MPCC-constraint qualification,
there are further single-headed arrows pointing to the inner square.
Results that follow directly from other equivalences have arrows with the label (implied).

\ifnum\FmtChoice=1
 \amssidewaysfigure
\else
\sidewaysfigure
\fi
    \begin{center}{
        \begin{tikzpicture}[ampersand replacement=\&]
            \matrix (m) [matrix of math nodes,nodes={draw,anchor=center,text centered,align=center,text width=3.3cm,rounded corners,minimum width=2.0cm, minimum height=1.5cm},column sep=6.5pc,row sep=10em]
            {
            |(A)| {\begin{array}{c} \textup{Branch-Problems} \\ \textup{Reformulation} \\ \textup{\eqref{eq:branch-e-anf}} \end{array}}  \&
            |(B)| {\begin{array}{c} \textup{Reformulation} \\ \textup{\eqref{eq:e-anf}} \end{array}}  \&
            |(C)|  {\begin{array}{c} \textup{Reformulation} \\ \textup{\eqref{eq:e-mpec}} \end{array}} \&
            |(D)|  {\begin{array}{c} \textup{Branch-Problems} \\ \textup{Reformulation} \\ \textup{\eqref{eq:branch-e-mpec}} \end{array}}   \\
            |(E)| {\begin{array}{c} \textup{Branch-Problems} \\ \textup{\eqref{eq:branch-anf}} \end{array}}  \&
            |(F)| {\begin{array}{c}  \textup{Abs-normal NLP} \\ \textup{\eqref{eq:i-anf}} \end{array}} \&
            |(G)| {\begin{array}{c} \textup{Counterpart} \\ \textup{MPCC} \\ \textup{\eqref{eq:i-mpec}} \end{array}}  \&
            |(H)| {\begin{array}{c} \textup{Branch-Problems} \\ \textup{\eqref{eq:branch-mpec}} \end{array}}   \\
            };
            \begin{scope}[>={LaTeX[width=2mm,length=2mm]},->,line width=.5pt]
                \tikzstyle{every node}=[font=\footnotesize]

                \path[dashed] ([xshift=2em]A.north) edge[bend left=20,<->]  node[above] {\cref{th:branch-gcq-e}} ([xshift=-2em]D.north);
                \path[] ([xshift=-2em]A.north) edge[bend left=25,<->]  node[above] {\cref{th:branch-acq-e}} ([xshift=2em]D.north);

                \path[] ([yshift=1em]A.east) edge[->] node[below]{\cref{th:e-branch-acq_akq}}([yshift=1em]B.west);
                \path[dashed] ([yshift=-1em]A.east) edge[->] node[below] {\cref{th:e-branch-gcq_gkq}} ([yshift=-1em]B.west);

                \path[] ([yshift=1em]B.east) edge[<->] node[below]{\cref{th:acq-akq-e}}([yshift=1em]C.west);
                \path[dashed] ([yshift=-1em]B.east) edge[<-] node[below] {\cref{th:gcq-gkq-e}} ([yshift=-1em]C.west);

                \path[] ([yshift=1em]C.east) edge[<-] node[below]{\cref{th:e-branch-acq_mpec-acq}}([yshift=1em]D.west);
                \path[dashed] ([yshift=-1em]C.east) edge[<-] node[below] {\cref{th:e-branch-gcq_mpec-gcq}} ([yshift=-1em]D.west);

                \path[] ([xshift=-2em]A.south) edge[<->] node[above,rotate=-90]{\cref{th:branch-acq}}([xshift=-2em]E.north);
                \path[dashed] ([xshift=2em]A.south) edge[<->] node[above,rotate=-90]{\cref{th:branch-gcq}}([xshift=2em]E.north);

                \path[] ([xshift=-2em]B.south) edge[<->] node[above,rotate=-90]{\cref{th:akq}}([xshift=-2em]F.north);
                \path[dashed] ([xshift=2em]B.south) edge[->] node[above,rotate=-90]{\cref{th:gkq}}([xshift=2em]F.north);

                \path[dashed] ([xshift=-2em]C.south) edge[->] node[above,rotate=90]{\cref{th:gcq}}([xshift=-2em]G.north);
                \path[] ([xshift=2em]C.south) edge[<->] node[above,rotate=90]{(implied)}([xshift=2em]G.north);

                \path[dashed] ([xshift=-2em]D.south) edge[<->] node[above,rotate=90]{(implied)}([xshift=-2em]H.north);
                \path[] ([xshift=2em]D.south) edge[<->] node[above,rotate=90]{(implied)}([xshift=2em]H.north);

                \path[dashed] ([yshift=1em]E.east) edge[->] node[below]{\cref{th:branch-gcq_gkq}}([yshift=1em]F.west);
                \path[] ([yshift=-1em]E.east) edge[->] node[below] {\cref{th:branch-acq_akq}} ([yshift=-1em]F.west);

                \path[dashed] ([yshift=1em]F.east) edge[<-] node[above] {\cref{th:gkq-gcq-i}} ([yshift=1em]G.west);
                \path[] ([yshift=-1em]F.east) edge[<->]  node[above] {\cref{th:akq-acq-i}} ([yshift=-1em]G.west);

                \path[dashed] ([yshift=1em]G.east) edge[<-] node[below]{\cref{th:branch-gcq_mpec-gcq}}([yshift=1em]H.west);
                \path[] ([yshift=-1em]G.east) edge[<-] node[below] {\cref{th:branch-acq_mpec-acq}} ([yshift=-1em]H.west);

                \path[dashed] ([xshift=2em]E.south) edge[bend right=20,<->]  node[above] {\cref{th:branch-gcq-i}} ([xshift=-2em]H.south);
                \path[] ([xshift=-2em]E.south) edge[bend right=25,<->]  node[above] {\cref{th:branch-acq-i}} ([xshift=2em]H.south);
            \end{scope}
        \end{tikzpicture}}
    \end{center}
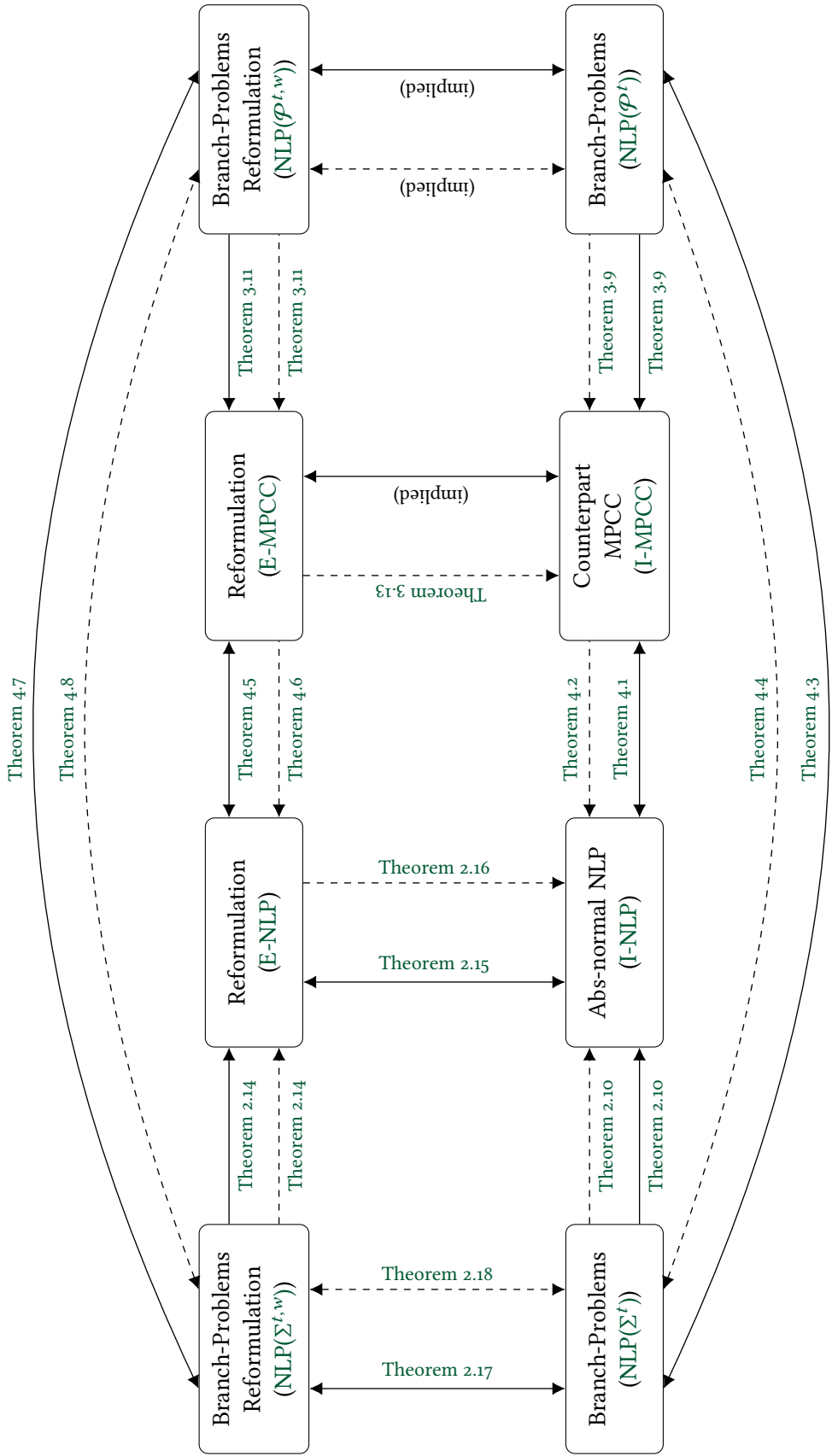
    \caption{Solid arrows: relations between AKQ and MPCC-ACQ;
      dashed arrows: relations between GKQ and MPCC-GCQ. Note that in \cref{th:gkq}, \cref{th:gcq}, \cref{th:gkq-gcq-i} and \cref{th:gcq-gkq-e} we have only proved one-sided implications and it is open whether the reverse implications hold.}
    \label{fig:relations}
\ifnum\FmtChoice=1
 \endamssidewaysfigure
\else
\endsidewaysfigure
\fi

\section{First Order Stationarity Concepts}
\label{sec:stationarity}

In this section, we introduce definitions of Mordukhovich stationarity
and Bouligand stationarity for abs-normal NLPs
and compare these definitions to M-stationarity and B-stationarity for MPCCs.
We give proofs based on the general formulation.

\subsection{Mordukhovich Stationarity}
In this paragraph we have a closer look at M-stationarity \cite{Outrata1999},
which is a necessary optimality condition for MPCCs under MPCC-ACQ \cite{FlegelKanzow2006}.

\begin{definition}[M-Stationarity for \eqref{eq:i-mpec}, see \cite{Outrata1999}]\label{def:m-stat}
  Consider a feasible point $(x^*,u^*,v^*)$ of \eqref{eq:i-mpec}
  with associated index sets $\setU_+$, $\setV_+$ and $\setD$.
  It is an \emph{M-stationary} point
  if there exist multipliers $\lambda = (\lame,\lami,\lamz)$ and
  $\mu=(\muu,\muv)$ such that the following conditions are satisfied:
  \begin{subequations}\label{eq:m-stat}
  \begin{align}
    \partial_{x,u,v} \Lc(x^*,u^*,v^*,\lambda,\mu) &= 0,\label{eq:m-stat-a}\\
    ((\muu)_i > 0,\ (\muv)_i > 0) \ \vee \ (\muu)_i(\muv)_i &= 0, \ i\in \setD\label{eq:m-stat-b}\\
    (\muu)_i &= 0, \ i\in \setU_+,\label{eq:m-stat-c}\\
    (\muv)_i &= 0, \ i\in \setV_+,\label{eq:m-stat-d}\\
    \lami &\ge 0,\label{eq:m-stat-e}\\
    \lami^T \ci(x^*,u^*,v^*) &= 0\label{eq:m-stat-f}.
  \end{align}
  \end{subequations}
  Herein, $\Lc$ is the MPCC-Lagrangian function
  \begin{align*}
    \Lc(x,u,v,\lambda,\mu)\define f(x)
    &+\lame^T\ce(x,u+v) - \lami^T\ci(x,u+v)\\
    &+\lamz^T[\cz(x,u+v) - (u-v)] - \muu^T \ut - \muv^T \vt.
  \end{align*}
\end{definition}

\ifCmp
Local minimizers of \eqref{eq:i-mpec}
are M-stationary points under MPCC-ACQ, as shown in \cite{FlegelKanzow2006,FlegelDiss}.
\begin{minipage}{0.74\textwidth}
The name M-stationarity was introduced by Scholtes in \cite{Scholtes_2001} and was motivated by the fact that
the sign restrictions on the multipliers in \eqref{eq:m-stat} in fact model the Mordukhovich normal cone.
The inset on the right illustrates the feasible set of multiplier values for a pair $((\mu_{{u}})_i,(\mu_{{v}})_i)$, $i\in\mathcal D$.
M-stationarity is a weaker stationarity concept than strong stationarity, but is at the same time the strongest
necessary optimality condition known to hold in absence of a strong constraint qualification like MPCC-MFCQ.
\end{minipage}
\hspace*{.5cm}\begin{minipage}{0.25\textwidth}
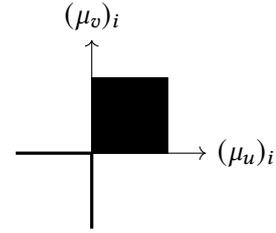
\begin{tikzpicture}
\draw[fill=black] (0,0)--(1,0)--(1,1)--(0,1)--(0,0);
\draw[<->] (1.5,0) node[right]{$(\mu_{{u}})_i$} --(0,0)--(0,1.5) node[above]{$(\mu_{{v}})_i$};
\draw[very thick] (-1,0)--(0,0)--(0,-1);
\end{tikzpicture}
\end{minipage}
\fi

\begin{definition}[M-Stationarity for \eqref{eq:i-anf}]\label{def:m-stat-anf}
  Consider a feasible point $(x^*,z^*)$ of \eqref{eq:i-anf}.
  It is an \emph{M-stationary} point
  if there exist multipliers $\lambda = (\lame,\lami,\lamz)$
  such that the following conditions are satisfied:
  \begin{subequations}\label{eq:m-stat-anf}
  \begin{align}
    f'(x^*) +
    \lame^T \partial_1 \ce -
    \lami^T \partial_1 \ci +
    \lamz^T \partial_1 \cz &= 0,\label{m-stat-anf-a}\\
    [\lame^T \partial_2 \ce -
    \lami^T \partial_2 \ci +
    \lamz^T \partial_2 \cz]_i &= (\lamz)_i\sigma^*_i, \ i\notin\alpha(x^*),\label{m-stat-anf-b}\\
    (\mu_i^-)(\mu_i^+)=0 \ \vee \ [\lame^T \partial_2 \ce -
    \lami^T \partial_2 \ci +
    \lamz^T \partial_2 \cz]_i &> \abs{(\lamz)_i}, \ i\in\alpha(x^*),\label{m-stat-anf-c}\\
    \lami &\ge 0, \label{m-stat-anf-d}\\
    \lami^T \ci &= 0\label{m-stat-anf-e}.
  \end{align}
  \end{subequations}
   Here we use the notation
    \begin{align*}
     \mu^+_i&\define\left[\lame^T \partial_2 \ce - \lami^T \partial_2 \ci +  \lamz^T [\partial_2 \cz-I]\right]_i,\\
     \mu^-_i&\define\left[\lame^T \partial_2 \ce - \lami^T \partial_2 \ci +  \lamz^T [\partial_2 \cz+I]\right]_i,
    \end{align*}
  and the constraints and the partial derivatives are evaluated
  at $(x^*,\abs{z^*})$.
\end{definition}

\begin{theorem}[M-Stationarity for \eqref{eq:i-mpec} is M-Stationarity for \eqref{eq:i-anf}]\label{thm:m-stat-is-m-stat}
A feasible point $(x^*,z^*)$ of \eqref{eq:i-anf} is M-stationary if and only if
$(x^*,u^*,v^*)=(x^*,[z^*]^+,[z^*]^-)$ of \eqref{eq:i-mpec} is M-stationary.
\end{theorem}
\begin{proof}
For indices that satisfy the first condition in \eqref{eq:m-stat-b},
the equivalence with the second condition in \eqref{m-stat-anf-c}
was shown in \cite[Theorem 33]{Hegerhorst_et_al:2019:MPEC2}.
Thus, we just need to consider the alternative conditions.
For \eqref{eq:i-mpec} we have the relations
  \begin{align*}
   \left[\lame^T \partial_2 \ce - \lami^T \partial_2 \ci + \lamz^T [\partial_2 \cz -I]\right]_i &= (\muu)_i,
   \ i\in \setD, \\
   \left[\lame^T \partial_2 \ce - \lami^T \partial_2 \ci + \lamz^T [\partial_2 \cz +I]\right]_i &= (\muv)_i,
   \ i\in \setD,
  \end{align*}
which was also shown in \cite[Theorem 33]{Hegerhorst_et_al:2019:MPEC2}.
These are exactly the definitions of $\mu_i^+$ and $\mu_i^-$ in the definition of
M-Stationarity for \eqref{eq:i-anf}.
\end{proof}

Consequently, we may now rephrase the result by \cite{FlegelKanzow2006,FlegelDiss} in the language of abs-normal forms.

\begin{theorem}[Minimizers and M-Stationarity for \eqref{eq:i-anf}]\label{th:local-min-m-stat}
Assume that $(x^*,z^*)$ is a local minimizer of \eqref{eq:i-anf} and that AKQ holds at $x^*$.
Then, $(x^*,z^*)$ is M-stationary for \eqref{eq:i-anf}.
\end{theorem}
\begin{proof}
  First, note that $(x^*,z^*)$ is a local minimizer of \eqref{eq:i-anf} if and only if
  $(x^*,u^*,v^*)=(x^*,[z^*]^+,[z^*]^-)$ is a local minimizer of \eqref{eq:i-mpec}.
  Then, the point $(x^*,u^*,v^*)$ is a local minimizer of the counterpart MPCC,
  and MPCC-ACQ holds by \cref{th:akq-acq-i}.
  \ifCmp Thus, \fi
  $(x^*,u^*,v^*)$ is M-stationary for \eqref{eq:i-mpec}
  and \cref{thm:m-stat-is-m-stat} implies
  that $(x^*,z^*)$ is M-stationary for \eqref{eq:i-anf}.
\end{proof}

\subsection{MPCC-linearized Bouligand Stationarity}
Finally, we introduce MPCC-linearized Bouligand stationarity, which is defined via smooth subproblems.

\begin{definition}[MPCC-linearized B-Stationarity for \eqref{eq:i-mpec}, see \cite{Scheel_Scholtes_2000}]\label{def:b-stat-mpec}
Consider a feasible point $(x^*,u^*,v^*)$ of \eqref{eq:i-mpec} with associated index sets $\setU_+$, $\setV_+$ and $\setD$.
It is a \emph{B-stationary} point if it is a stationary point of all
branch problems \eqref{eq:branch-mpec} for $\setPt=\setP \subseteq \setD$.
Here, $\bar\setP$ denotes the complement of $\setP$ in $\setD(x^*)$.
\end{definition}

Note that there exist different names for the variant of B-stationarity just introduced. It is simply called \emph{B-stationarity} in \cite{Scheel_Scholtes_2000},
but we prefer here the name \emph{MPCC-linearized B-stationarity} suggested in \cite{FlegelDiss} to prevent confusion with the definition of B-stationarity in the smooth case. The concept of B-stationarity is the most intuitive among stationarity concepts in simply requiring that, no matter how degenerate pairs of complementarities are resolved, no first order descent direction may be revealed. Moreover, it may be brought into agreement with the concept of local minimizers already under very weak assumptions regarding constraint qualifications. The downside however is that verifying B-stationarity inherently requires exponential runtime effort, as the number of branch problems is exponential in the number of degenerate pairs in $\setD$.

\begin{theorem}\label{th:local-min-b-stat_mpec}
 If GCQ holds for all \eqref{eq:branch-mpec}, then all local minimizers of \eqref{eq:i-mpec} are MPCC-linearized B-stationary points.
\end{theorem}
\begin{proof}
 This follows directly by KKT theory for smooth optimization problems.
\end{proof}

\begin{definition}[Abs-Normal-Linearized B-Stationarity for \eqref{eq:i-anf}]\label{def:b-stat-anf}
Consider a feasible point $(x^*,z^*)$ of \eqref{eq:i-anf}.
It is an \emph{abs-normal-linearized B-stationary} point if it is a stationary point of the
branch problems \eqref{eq:branch-anf} for $\Sigt=\diag(\sigma)$ with $\sigma \succeq \sigma(x)$.
\end{definition}

\begin{theorem}[MPCC-linearized B-stationarity for \eqref{eq:i-mpec} is abs-normal-linearized B-stationarity for \eqref{eq:i-anf}]\label{th:b-stat-is-b-stat}
A feasible point $(x^*,z^*)$ of \eqref{eq:i-anf} is abs-normal-linearized B-stationary if and only if
$(x^*,u^*,v^*)=(x^*,[z^*]^+,[z^*]^-)$ of \eqref{eq:i-mpec} is MPCC-linearized B-stationary.
\end{theorem}
\begin{proof}
Every branch problem \eqref{eq:branch-anf} is smooth and thus stationarity is equivalent to the condition $f'(x^*)^Td\ge 0$ for all $d\in\Tlinsig(x^*,z^*)$.
Analogously, stationarity for every branch problem \eqref{eq:branch-mpec} is equivalent to the condition $f'(x^*)^Td\ge 0$ for all $d\in\Tlinp(x^*,[z^*]^+,[z^*]^-)$.
Then, the equivalence follows as both branch problems are homeomorphic
and both linearization cones are homeomorphic by \cref{le:hom-T-branch-i}.
\end{proof}

\begin{theorem}[Minimizers and abs-normal-linearized B-Stationarity for \eqref{eq:i-anf}]\label{th:local-min-b-stat}
Assume that $(x^*,z^*)$ is a local minimizer of \eqref{eq:i-anf} and that GCQ holds at $(x^*,z^*)$ for all \eqref{eq:branch-anf}.
Then, $(x^*,z^*)$ is abs-normal-linearized B-stationary for \eqref{eq:i-anf}.
\end{theorem}
\begin{proof}
  The point $(x^*,z^*)$ is a local minimizer of \eqref{eq:i-anf} if and only if
  $(x^*,u^*,v^*)=(x^*,[z^*]^+,[z^*]^-)$ is a local minimizer of \eqref{eq:i-mpec}.
  Moreover, GCQ for all \eqref{eq:branch-anf} and GCQ for all \eqref{eq:branch-mpec} are equivalent by \cref{th:branch-gcq-i}.
  Thus, $(x^*,u^*,v^*)$ is a local minimizer of the counterpart MPCC and GCQ holds for all \eqref{eq:branch-mpec}.
  Then, it is MPCC-linearized B-stationary by \cref{th:local-min-b-stat_mpec} and finally
  $(x^*,z^*)$ is abs-normal-linearized B-stationary by \cref{th:b-stat-is-b-stat}.
\end{proof}

\begin{remark}
 In \cite{Griewank_Walther_2019}, Griewank and Walther have presented
 a stationarity concept that holds without any kink qualification
 for minimizers of the \emph{unconstrained} abs-normal NLP
 \begin{equation}
   \label{eq:anf-0}
   \min_{x} f(x), \quad f\in \Cabs(\Domx,\R).
 \end{equation}
 Indeed, this concept is precisely abs-normal-linearized Bouligand stationarity:
 it requires the conditions of \cref{def:b-stat-anf}
 specialized to \eqref{eq:anf-0}.
 Now, the question arises why no regularity assumption is needed.
 The answer is that the abs-normal form provides
 a certain degree of built-in regularity:
 we have shown in \cite{Hegerhorst_et_al:2019:MPEC1}
 that MPCC-ACQ is always satisfied for the counterpart MPCC of \eqref{eq:anf-0} (and thus every local minimizer is an M-stationary point).
 Analogously one can show that ACQ for all branch problems \eqref{eq:branch-mpec} is always satisfied for \eqref{eq:anf-0}.
 Now, ACQ for all branch problems \eqref{eq:branch-mpec} is equivalent to ACQ for all branch problems \eqref{eq:branch-anf}
 by \cref{th:branch-acq-i}, which in turn implies GCQ for all branch problems \eqref{eq:branch-anf}.
 Thus, GCQ for all branch problems \eqref{eq:branch-anf} is always satisfied for \eqref{eq:anf-0} and \cref{th:local-min-b-stat} holds.
\end{remark}

\section{Conclusions}
\label{sec:conclusions}

We have shown that general abs-normal NLPs
are essentially the same problem class as MPCCs.
The two problem classes permit the definition of corresponding constraint qualifications,
and optimality conditions of first order under weak constraint qualifications.
We have also shown that the slack reformulation
from \cite{Hegerhorst_Steinbach:2019}
preserves constraint qualifications of Abadie type, whereas for Guginard type we could only prove some implications.
Here, one subtle drawback is the non-uniqueness of slack variables.
Thus, we have introduced branch formulations of general abs-normal NLPs and counterpart MPCCs.
Then, constraint qualifications of Abadie and Guignard type are preserved.

\ifcase\FmtChoice
\bibliographystyle{tfs}
\or
\bibliographystyle{tfs}
\or
\bibliographystyle{jnsao}
\fi
\bibliography{abs_normal_nlp}

\end{document}